\documentclass[11pt]{amsart}
\usepackage{amsmath,amssymb,txfonts}
\usepackage{amssymb}
\usepackage{amsmath}
\usepackage{mathrsfs}
\usepackage{amsmath,amssymb}
\usepackage{color}

\newtheorem{theorem}{Theorem}[section]
\newtheorem{lemma}[theorem]{Lemma}
\newtheorem{proposition}[theorem]{Proposition}
\newtheorem{corollary}[theorem]{Corollary}
\theoremstyle{definition}
\newtheorem{definition}[theorem]{Definition}

\theoremstyle{remark}
\newtheorem{remark}[theorem]{Remark}

\numberwithin{equation}{section}

\begin{document}

\title[ Generalized Morrey spaces]
{Schr\"odinger type operators on generalized Morrey spaces}

\author[Pengtao Li]{Pengtao\ Li}
\address{College of Mathematics, Qingdao University, Qingdao, Shandong 266071, China}

\email{ptli@qdu.edu.cn}

\author[Xin Wan]{Xin Wan}
\address{Department of Mathematics, Shantou University, Shantou, Guangdong 515063, China.}
\email{13xwan@stu.edu.cn}

\author[Chuangyuan Zhang]{Chuangyuan Zhang}
\address{Department of Mathematics, Shantou University, Shantou, Guangdong 515063, China.}
\email{12cyzhang@stu.edu.cn}



\date{}

\dedicatory{}


\begin{abstract}
In this paper we introduce a class of generalized Morrey spaces associated with Schr\"odinger operator $L=-\Delta+V$. Via a pointwise estimate, we obtain the boundedness of the operators $V^{\beta_{2}}(-\Delta+V)^{-\beta_{1}}$ and their dual operators on these Morrey spaces.
\end{abstract}

\maketitle \tableofcontents \pagenumbering{arabic}
\section{Introduction}

The investigation of Schr\"odinger operators on the Euclidean space
$\mathbb{R}^{n}$ with nonnegative potentials which belong to
the reverse H\"older class has attracted attention of many authors. Shen \cite{Sh}  studied the Schr\"odinger
operator $L=-\Delta+ V$, assuming the nonnegative potential $V$ belongs to the reverse H\"older class $B_{q}, q\geq\frac{n}{2}$.
In \cite{Sh}, Shen proved the $L^{p}$-boundedness of the operators $(-\Delta+V)^{i\gamma}$, $\nabla^{2}(-\Delta+V)^{-1}$, $\nabla(-\Delta+V)^{-1/2}$ and $\nabla(-\Delta+V)^{-1}\nabla$.
For further information, we refer the reader to Guo-Li-Peng \cite{GLP}, Liu \cite{LiuY-1}, Liu-Huang-Dong \cite{LHD}, Liu-Wang-Dong \cite{LWD}, Tang-Dong \cite{TD}, Yang-Yang-Zhou \cite{YYZ, YYZ2} and the references therein.

The purpose of this paper is to generalize the results of  Shen \cite{Sh} and Sugano \cite{Su} to a class of  Morrey spaces associated with $L$, denoted by $L_{\alpha,\theta,V}^{p,q,\lambda}(\mathbb{R}^{n})$. See Definition \ref{def-HMorrey} below. The significance of these spaces is that for particular choices of the
parameters $p$, $q$, $\lambda$, $\theta$ and $\alpha$, one obtains many classical function spaces. In particular,
\begin{table}[h]
\centering
\begin{tabular}{|l|c|} \hline

 $\theta=0, \alpha=0, p=q$, $0<\lambda<1$,
  & Morrey space $L^{p,\lambda}(\mathbb{R}^{n})$ \cite{P}\\ \hline
  $\theta=0$, $p=q$, $0<\lambda<1$,
  & Morrey type space $L_{\alpha,V}^{p,\lambda}(\mathbb{R}^{n})$ \cite{TD} \\ \hline
  $\alpha=\lambda=0$,  $\theta\in\mathbb{R}$, $0<p,q<\infty$,
  & Herz spaces $K_{p}^{\theta,q}$ \cite{Herz} \\ \hline
  $\alpha=0, \lambda\geq0$,  $\theta\in\mathbb{R}$, $0<p,q<\infty$,
  & Morrey-Herz spaces $MK_{p,q}^{\theta,\lambda}$ \cite{CC, SL} \\ \hline
\end{tabular}
\end{table}\par

In Section \ref{sec2}, let $T$ be one of the Schrodinger type operators $\nabla(-\Delta+V)^{-1}\nabla$, $\nabla(-\Delta+V)^{-1/2}$ and $(-\Delta+V)^{-1/2}\nabla$. With the help of the $L^{p}$-boundedness of $T$, it is easy to verify that $T$ is bounded on $L_{\alpha,\theta, V}^{p,q, \lambda}(\mathbb{R}^{n})$. For $b\in BMO(\mathbb{R}^{n})$, we can also obtain the boundedness of the commutator $[b, T]$ on $L_{\alpha,\theta, V}^{p,q, \lambda}(\mathbb{R}^{n})$. See Theorem \ref{th-CZ} \& \ref{th-commutator}. For  $\theta=0$, $p=q$ and $0<\lambda<1$, $L_{\alpha,0, V}^{p,p, \lambda}(\mathbb{R}^{n})$ becomes the spaces $L_{\alpha,V}^{p,\lambda}(\mathbb{R}^{n})$ introduced by Tang-Dong \cite{TD}. Hence, the results are  generalizations of \cite[Theorems 1 \& 2]{TD}.

In recent years, the fractional integral operator $I_{\alpha}=(-\Delta+V)^{-\alpha}$ has been studied extensively. We refer to Duong-Yan \cite{DY}, Jiang \cite{J}, Tang-Dong \cite{TD} and Yang-Yang-Zhou \cite{YYZ} for details. Suppose that $V\in B_{s}, s\geq\frac{n}{2}$. For $0\leq\beta_{2}\leq\beta_{1}<\frac{n}{2}$, let
\begin{equation}\nonumber
\begin{cases}
T_{\beta_{1},\beta_{2}}=:V^{\beta_{2}}(-\Delta+V)^{-\beta_{1}},\\
T^{\ast}_{\beta_{1},\beta_{2}}=:(-\Delta+V)^{-\beta_{1}}V^{\beta_{2}}.
\end{cases}
\end{equation}
Sugano \cite{Su} obtained the weighted estimates for $T_{\beta_{1}, \beta_{2}}, T^{\ast}_{\beta_{1},\beta_{2}}, 0<\beta_{2}\leq\beta_{1}<1$.
If $\beta_{2}=0$, we can see that $T_{\beta_{1},0}=I_{\beta_{1}}$. So $T_{\beta_{1},\beta_{2}}$ and $T_{\beta_{1},\beta_{2}}^{\ast}$ can be seen as generalizations of $I_{\alpha}$. Moreover, for $(\beta_{1}, \beta_{2})=(1,1)$ and $(1/2,1/2)$, $T_{1,1}^{\ast}=(-\Delta+V)^{-1}V$ and $T_{1/2,1/2}^{\ast}=(-\Delta+V)^{-1/2}V^{1/2}$, respectively, which are studied by Shen \cite{Sh} thoroughly.
In Section \ref{sec3}, assume that $1<p_{1}<\infty$, $1<p_{2}<{s}/{\beta_{2}}$ and $1<q<\infty$. If the index $(q, \beta_{1},\beta_{2},\lambda,\alpha,\theta)$ satisfies:
\begin{equation}\nonumber
\begin{cases}
{1}/{p_{2}}={1}/{p_{1}}-{2(\beta_{1}-\beta_{2})}/{n},\\
\alpha\in(-\infty,0]\ \&\ \lambda\in(0,n),\\
{\lambda}/{q}-{1}/{p_{1}}+{2\beta_{1}}/{n}<\theta<{\lambda}/{q}+1-{1}/{p_{1}},
\end{cases}
\end{equation}
   we prove that  $T_{\beta_{1},\beta_{2}}$ is bounded from $L^{p_{1},q,\lambda}_{\alpha,\theta, V}(\mathbb{R}^{n})$ to $L^{p_{2},q,\lambda}_{\alpha,\theta,V}(\mathbb{R}^{n}).$
Specially, we know that $(-\Delta+V)^{-1}V$ and $(-\Delta+V)^{-1/2}V^{1/2}$ are bounded on $L_{\alpha,\theta, V}^{p,q, \lambda}(\mathbb{R}^{n})$. See Theorems \ref{th-operator} \& \ref{th-dual-operator} for the details.

In the research of harmonic analysis and partial differential equations, the commutators play an important role. If $T$ is a Calder\'on-Zygmund operator, $b\in BMO(\mathbb{R}^{n})$, the $L^{p}-$boundedness of $[b,T ]$ was first
discovered by Coifman-Rochberg-Weiss \cite{CRW}. Later, Str\"omberg [4] gave a simple proof, adopting the idea of
relating commutators with the sharp maximal operator of Fefferman and Stein. In 2008, Guo-Li-Peng \cite{GLP} introduced a condition $H(m)$ and obtained that $L^{p}-$boundedness of the commutator of Riesz transforms associated with $L$, where $b\in BMO(\mathbb{R}^{n})$. For further information, we refer to Liu \cite{Liu2}, Liu-Huang-Dong \cite{LHD}, Liu-Wang-Dong \cite{LWD}, Yang-Yang-Zhou \cite{YYZ2} and the references therein.

In Section \ref{sec-5}, by the boundedness of $I_{\alpha}$ and $(-\Delta+V)^{-\beta}V^{-\beta}$, we can deduce that the commutators $[b, T_{\beta_{1},\beta_{2}}]$ and $[b, T^{\ast}_{\beta_{1},\beta_{2}}]$
are bounded from $L^{p_{1}}(\mathbb{R}^{n})$ to $L^{p_{2}}(\mathbb{R}^{n})$. See Theorem \ref{th-bdd-Tbeta}.  Theorem \ref{th-bdd-Tbeta} together with Lemmas \ref{le3.2} \& \ref{le3.3} can be used to prove that the commutators $[b, T_{\beta_{1},\beta_{2}}]$ and $[b, T^{\ast}_{\beta_{1},\beta_{2}}]$ are bounded from $L_{\alpha,\theta,V}^{p_{1},q,\lambda}(\mathbb{R}^{n})$ to $L_{\alpha,\theta,V}^{p_{2},q,\lambda}(\mathbb{R}^{n})$, respectively. See Theorems \ref{th-5.2} \& \ref{th-5.3}.
\begin{remark}
Unlike the setting of the Lebesgue spaces, it is well-known that the dual of $L^{p,\lambda}(\mathbb{R}^{n})$ is not $L^{p',-\lambda}(\mathbb{R}^{n})$.   Hence, after obtaining Theorem \ref{th-operator}, we can not deduce Theorem \ref{th-dual-operator} via the method of duality used by Guo-Li-Peng \cite{GLP}.
\end{remark}

\section{Preliminaries}

\subsection{Schr\"odinger operator and the auxiliary function}In this paper, we consider the Schr\"{o}dinger differential operator $L=-\Delta+V$ on $\mathbb{R}^{n}$, $n\geq 3$, where $V$ is a nonnegative potential belonging to the reverse H\"{o}lder class $B_{s}, s\geq\frac{n}{2},$ whch is defined as follows.
\begin{definition} Let $V$ be a nonnegative function.
\item{(i)} We say  $V\in B_{s}, s>1$, if there exists $C>0$ such that for every ball $B\subset\mathbb{R}^{n}$, the reverse H\"{o}lder inequality
$$\left(\frac{1}{|B|}\int_{B}V^{s}(x)dx\right)^{\frac{1}{s}}\lesssim\left(\frac{1}{|B|}\int_{B}V(x)dx\right)$$
holds.
\item{(ii)} We say $V\in B_{\infty}$ if there exists a constant $C$ such that for every ball $B\subset \mathbb{R}^{n}$,
$$\|V\|_{L^{\infty}(B)}=\frac{1}{|B|}\int_{B}V(x)dx.$$
\end{definition}
\begin{remark}
Assume $V\in B_{s}$, $1<s<\infty$. Then $V(y)dy$ is a doubling measure. Namely, there exists a constant $C_{0}$ such that for any $r>0$ and $y\in \mathbb{R}^{n}$,
\begin{equation}\label{eq-double}
\int_{B(x,2r)}V(y)dy\lesssim C_{0}\int_{B(x,r)}V(y)dy.
\end{equation}
\end{remark}

\begin{definition}(Shen \cite{Sh})
For $x\in\mathbb{R}^{n}$, the function $m_{V}(x)$ is defined as
\begin{eqnarray*}
\frac{1}{m_{V}(x)}=:\sup\left\{ r>0 :\frac{1}{r^{n-2}}\int_{B(x,r)}V(y)dy\leq 1\right\}.
\end{eqnarray*}
\end{definition}
\begin{remark}
The function $m_{V}$ reflects the scale of $V$ essentially, but behaves better. It is deeply studied in Shen \cite{Sh} and  play a crucial role in our proof. We list a property of $m_{V}$ which will be used in the sequel and refer the reader to Guo-Li-Peng \cite {GLP} for the details.
\end{remark}

We state some notations and properties of  $m_{V}$.
\begin{lemma}\label{le2.1}
{\rm (\cite[Lemma 1.4]{Sh})} Suppose that $V \in B_{s}$ with $s\geq \frac{n}{2}$. Then there exist positive constants $C$ and $k_{0}$ such that
\begin{itemize}
\item[(a)]  if $|x-y|\leq \frac{C}{m_{V}(x)}$, $m_{V}(x)\sim m_{V}(y)$;
\item[(b)] $m_{V}(y)\lesssim(1+|x-y| m_{V}(x))^{k_{0}}m_{V}(x);$
\item[(c)] $m_{V}(y)\geq {Cm_{V}(x)}/\{1+|x-y|m_{V}(x)\}^{k_{0}/(k_{0}+1)}.$
\end{itemize}
\end{lemma}

\begin{lemma}\label{le3.3.0}
{\rm (\cite[Lemma 1.2]{Sh})} Suppose that $V\in B_{s}, s>\frac{n}{2}$. There exists a constant $C$ such that for $0<r<R<\infty$,
$$\frac{1}{r^{n-2}}\int_{B(x,r)}V(y)dy\lesssim\Big(\frac{R}{r}\Big)^{\frac{n}{s}-2}\cdot\frac{1}{R^{n-2}}\int_{B(x,R)}V(y)dy.$$
\end{lemma}
\begin{lemma}\label{le3.3}
{\rm (\cite[Lemma 2.3]{GLP})} Suppose $V\in B_{s}, s>\frac{n}{2}$. Then for any $N>\log_{2}C_{0}+1$, there exists a constant $C_{N}$ such that for any $x\in\mathbb{R}^{n}$ and $r>0$,
$$\frac{1}{(1+rm_{V}(x))^{N}}\int_{B(x,r)}V(y)dy\lesssim C_{N}r^{n-2}.$$
\end{lemma}

\subsection{Generalized Morrey spaces associated with $L$}  Suppose that $V\in B_{s}$, $s>1$.  Let $L=-\Delta+V$ be the Schr\"odinger operator. Now we introduce a class of generalized Morrey spaces associated with  $L$.
For $k\in \mathbb{Z}$, let $E_{k}=B(x_{0},2^{k}r)\backslash B(x_{0},2^{k-1}r)$ and $\chi_{k}$  be the characteristic function of $E_{k}$.
\begin{definition}\label{def-HMorrey}
Suppose that $V\in B_{s}$, $s>1$.  Let $p\in [1,+\infty)$, $q\in [1,+\infty)$, $\alpha\in (-\infty,+\infty)$ and $\lambda\in (0,n)$, $\theta\in (-\infty,+\infty)$. For $f\in L_{loc}^{q}(\mathbb{R}^{n})$, we say $f\in L_{\alpha,\theta,V}^{p,q,\lambda}(\mathbb{R}^{n})$ provided that
$$\|f\|^{q}_{L_{\alpha,\theta,V}^{p,q,\lambda}(\mathbb{R}^{n})}=\sup_{B(x_{0},r)\subset \mathbb{R}^{n}}\frac{(1+rm_{V}(x_{0}))^{\alpha}}{r^{\lambda n}}\sum^{0}_{k=-\infty}|E_{k}|^{\theta q}\|\chi_{k}f\|^{q}_{L^{p}(\mathbb{R}^{n})}<\infty,$$
where $B(x_{0},r)$ denotes a ball  centered at $x_{0}$ and with radius r.
\end{definition}

\begin{proposition}
\item{\rm (i)}  For $\alpha_{1}>\alpha_{2}$, $L_{\alpha_{1},\theta, V}^{p,q,\lambda}(\mathbb{R}^{n})\subseteq L_{\alpha_{2},\theta, V}^{p,\lambda,q}(\mathbb{R}^{n})$;
\item{\rm (ii)} If $\theta=0$, $p=q$ and $\alpha<0$, $L^{p,\lambda}(\mathbb{R}^{n}) \subset L_{\alpha,\theta,V}^{p,q, \lambda}(\mathbb{R}^{n})$;
 \item{\rm (iii)}If $\theta=0$, $p=q$ and $\alpha>0$, $L_{\alpha,\theta,V}^{p,q, \lambda}(\mathbb{R}^{n}) \subset L^{p,\lambda}(\mathbb{R}^{n})$.
\end{proposition}

\subsection{Calder\'on-Zygmund operators}We say that an operator $T$ taking $C_{c}^{\infty}(\mathbb{R}^{n})$ into $L_{loc}^{1}(\mathbb{R}^{n})$ is called a Calder\'{o}n-Zygmund operator if
\begin{itemize}
\item[(a)] $T$ extends to a bounded linear operator on $L^{2}(\mathbb{R}^{n})$;
\item[(b)] There exists a kernel $K$ such that for every $f\in L_{loc}^{1}(\mathbb{R}^{n})$;
$$Tf(x)=\int_{\mathbb{R}^{n}}K(x,y)f(y)dy \ \  a.e.\ on\ \{\text{ supp }f\}^{c},$$
\item[(c)] The kernel $K(x,y)$ satisfies the Calder\'{o}n-Zygmund estimate
$$|K(x,y)|\leq \frac{C}{|x-y|^{n}};$$
$$|K(x+h,y)-K(x,y)|+|K(x,y+h)-K(x,y)|\leq \frac{C|h|^{\delta}}{|x-y|^{n+\delta}}$$
\end{itemize}
for $x,y\in \mathbb{R}^{n}$, $|h|<\frac{|x-y|}{2}$ and for some $\delta>0$.

Shen \cite{Sh} obtained the following result.
\begin{theorem}\label{theorem-CZ}{\rm \cite[Theorem 0.8]{Sh})}
Suppose $V\in B_{n}$. Then
$$\nabla(-\Delta+V)^{-1}\nabla, \nabla(-\Delta+V)^{-\frac{1}{2}}\text{ and }(-\Delta+V)^{-\frac{1}{2}}\nabla$$
are Calder\'{o}n-Zygmund operators.
\end{theorem}
\begin{corollary}\label{coro-commutator-bdd}
Suppose that $V\in B_{n}$ and $b\in BM(\mathbb{R}^{n})$. The commutator $[b, T]$ is bounded on $L^{p}(\mathbb{R}^{n})$.
\end{corollary}
 In particular, let $K$ denote the kernel of one of the above operators. Then $K$ satisfies the following estimate:
\begin{equation}\label{eq2.1}
|K(x,y)|\leq \frac{C_{N}}{(1+\mid x-y\mid m_{V}(x))^{N}}\frac{1}{|x-y|^{n}}
\end{equation}
for any $N\in\mathbb{N}$. See (6.5) of Shen \cite{Sh} for the details.

Suppose $V\in B_{s}$ for $s\geq\frac{n}{2}$.  Let $L=-\Delta+V$. The semigroup generated by $L$ is defined as:
\begin{equation}\label{eq3.1}
T_{t}f(x)=e^{-tL}f(x)=\int_{\mathbb{R}^{n}}K_{t}(x,y)f(y)dy, f\in L^{2}(\mathbb{R}^{n}),\ t>0,
\end{equation}
where $K_{t}$ is the kernel of $e^{-tL}$.

\begin{lemma}\label{le3.1}
{\rm (\cite{DZ})} Let $K_{t}(x,y)$ be as in (\ref{eq3.1}). For every nonnegative integer k, there is a constant $C_{k}$ such that
$$0\leq K_{t}(x,y)\leq C_{k}t^{-\frac{n}{2}}\exp(-{\mid x-y \mid^{2}}/{5t})(1+\sqrt{t}\ m_{V}(x)+\sqrt{t}\ m_{V}(y))^{-k}.$$
\end{lemma}

\noindent{\bf Some notations.}Throughout the paper, $c$ and $C$ will denote unspecified positive constants, possibly different at each occurence. The constants are independence of the functions. ${\mathsf U}\approx{\mathsf V}$ represents that
there is a constant $c>0$ such that $c^{-1}{\mathsf V}\le{\mathsf
U}\le c{\mathsf V}$ whose right inequality is also written as
${\mathsf U}\lesssim{\mathsf V}$. Similarly, if ${\mathsf V}\ge
c{\mathsf U}$, we denote ${\mathsf V}\gtrsim{\mathsf U}$.

\section{Riesz transforms and the commutators on $L^{p,q,\lambda}_{\alpha,\theta,V}(\mathbb{R}^{n})$}\label{sec2}
Throughout this paper, for $p\in (1, \infty)$, denote by $p'$ the conjugate of $p$, that is, $\frac{1}{p}+\frac{1}{p'}=1.$ Let $V\in B_{n}$. In this section, we assume that $T$ is one of the Schrodinger type operators $\nabla(-\Delta+V)^{-1}\nabla$, $\nabla(-\Delta+V)^{-1/2}$ and $(-\Delta+V)^{-1/2}\nabla$.
  We study the boundedness on  $L^{p,q,\lambda}_{\alpha,\theta,V}(\mathbb{R}^{n})$ of $T$ and its commutator $[b,\ T]$ with $b\in BMO(\mathbb{R}^{n})$. The bounded mean oscillation space $BMO(\mathbb{R}^{n})$ is defined as follows.
 \begin{definition}
A locally integrable function $b$ is said to belong to $BMO(\mathbb{R}^{n})$ if
 \begin{equation*}
\|b\|_{BMO}=:\sup_{B}\frac{1}{|B|}\int_{B}|b(x)-b_{B}|dx<\infty,
\end{equation*}
where the supremum is taken over all balls $B$ in $\mathbb{R}^{n}$. Here $b_{B}=\frac{1}{|B|}\int_{B}b(x)dx$ stands for the mean value of $b$ over the ball $B$ and $|B|$ means the measure of $B$.
\end{definition}
  We first prove that $T$ is bounded on $L^{p,q,\lambda}_{\alpha,\theta,V}(\mathbb{R}^{n})$.

\begin{theorem}\label{th-CZ}
Suppose that $\alpha \in (-\infty,0]$, $\lambda \in (0,n)$ and $1<q<\infty$.
If $1<p<\infty,\frac{\lambda}{q}-\frac{1}{p}<\theta<\frac{\lambda}{q}+1-\frac{1}{p}$, then the operators $T$ are bounded on $L^{p,q,\lambda}_{\alpha,\theta,V}(\mathbb{R}^{n})$.
\end{theorem}
\begin{proof}
For any ball $B(x_{0},r)$, write
$$f(y)=\sum^{\infty}_{j=-\infty}f(y)\chi_{j}(y)=\sum^{\infty}_{j=-\infty}f_{j}(y),$$
where $E_{j}=B(x_{0},2^{j}r)\backslash B(x_{0},2^{j-1}r)$. Hence, we have

\begin{eqnarray}
&&{(1+rm_{V}(x_{0}))^{\alpha}}{r^{\lambda n}}\sum^{0}_{k=-\infty}|E_{k}|^{\theta q}\|\chi_{k}Tf\|^{q}_{L^{p}(\mathbb{R}^{n})}\nonumber\\
&&\lesssim\ {(1+rm_{V}(x_{0}))^{\alpha}}{r^{-\lambda n}}\sum^{0}_{k=-\infty}|E_{k}|^{\theta q}\left(\sum^{k-2}_{j=-\infty}\|\chi_{k}Tf_{j}\|_{L^{p}(\mathbb{R}^{n})}\right)^{q}\nonumber\\
&&\quad + {(1+rm_{V}(x_{0}))^{\alpha}}{r^{-\lambda n}}\sum^{0}_{k=-\infty}|E_{k}|^{\theta q}\left(\sum^{k+1}_{j=k-1}\|\chi_{k}Tf_{j}\|_{L^{p}(\mathbb{R}^{n})}\right)^{q}\nonumber\\
&&\quad+ {(1+rm_{V}(x_{0}))^{\alpha}}{r^{-\lambda n}}\sum^{0}_{k=-\infty}|E_{k}|^{\theta q}\left(\sum^{\infty}_{j=k+2}\|\chi_{k}Tf_{j}\|_{L^{p}(\mathbb{R}^{n})}\right)^{q}\nonumber\\
&&= A_{1}+A_{2}+A_{3}.\nonumber
\end{eqnarray}
For $A_{2}$, by Theorem \ref{theorem-CZ}, we have
\begin{eqnarray*}
A_{2}&\lesssim& (1+rm_{V}(x_{0}))^{\alpha}r^{-\lambda n}\sum^{0}_{k=-\infty}|E_{k}|^{\theta q}\left(\sum^{k+1}_{j=k-1}\|Tf_{j}\|_{L^{p}(\mathbb{R}^{n})}\right)^{q}\\
&\lesssim& (1+rm_{V}(x_{0}))^{\alpha}r^{-\lambda n}\sum^{0}_{k=-\infty}|E_{k}|^{\theta q}\left(\sum^{k+1}_{j=k-1}\|f_{j}\|_{L^{p}(\mathbb{R}^{n})}\right)^{q}\\
&\lesssim& \|f\|^{q}_{L^{p,q,\lambda}_{\alpha,\theta, V}}.
\end{eqnarray*}
We first estimate the term $E_{1}$.  Note that if $x\in E_{k}$, $y\in E_{j}$ and $j\leq k-2$, then $|x-y|\sim 2^{k}r$. By Lemma \ref{le2.1} and (\ref{eq2.1}), we can get
\begin{eqnarray*}
\|\chi_{k}Tf_{j}\|_{L^{p}(\mathbb{R}^{n})}
&\lesssim&\left(\int_{E_{k}}\mid\int_{\mathbb{R}^{n}}\frac{1}{(1+\mid x-y\mid m_{V}(x))^{N}}\frac{1}{|x-y|^{n}}|f_{j}(y)|dy\mid^{p}dx\right)^{\frac{1}{p}}\\
&\lesssim&\frac{1}{(1+2^{k}rm_{V}(x_{0}))^{N/k_{0}+1}}\frac{1}{(2^{k}r)^{n}}|E_{k}|^{\frac{1}{p}}\int_{E_{j}}|f(y)|dy\\
&\lesssim&\frac{1}{(1+2^{k}rm_{V}(x_{0}))^{N/k_{0}+1}}|E_{k}|^{\frac{1}{p}-1}|E_{j}|^{\frac{1}{p'}}\left(\int_{E_{j}}|f(y)|^{p}dy\right)^{\frac{1}{p}},
\end{eqnarray*}
where $\frac{1}{p}+\frac{1}{p'}=1$.
Since $- \frac{1}{p}+ \frac{\lambda}{q} <\theta<(1- \frac{1}{p})+ \frac{\lambda}{q}$, we obtain
\begin{eqnarray*}
A_{1}&\lesssim&(1+rm_{V}(x_{0}))^{\alpha}r^{-\lambda n}\sum^{0}_{k=-\infty}|E_{k}|^{\theta q}\left(\sum^{k-2}_{j=-\infty}\frac{|E_{k}|^{\frac{1}{p}-1}|E_{j}|^{\frac{1}{p'}}\|\chi_{j}f\|_{L^{p}(\mathbb{R}^{n})}}{(1+2^{k}rm_{V}(x_{0}))^{N/k_{0}+1}}\right)^{q}\\
&\lesssim&(1+rm_{V}(x_{0}))^{\alpha}r^{-\lambda n}\sum^{0}_{k=-\infty}|E_{k}|^{\theta q}\Big(\sum^{k-2}_{j=-\infty}\frac{2^{\frac{n(j-k)}{p'}}(1+2^{j}rm_{V}(x_{0}))^{-\frac{\alpha}{q}}}{(1+2^{k}rm_{V}(x_{0}))^{N/k_{0}+1}}\\
&&\times(2^{j}r)^{\frac{\lambda n}{q}}|E_{j}|^{-\theta}(1+2^{j}rm_{V}(x_{0}))^{\frac{\alpha}{q}}(2^{j}r)^{-\frac{\lambda n}{q}}(|E_{j}|^{\theta q}\|\chi_{j}f\|^{q}_{L^{p}(\mathbb{R}^{n})})^{\frac{1}{q}}\Big)^{q}\\
&\lesssim&(1+rm_{V}(x_{0}))^{\alpha}r^{-\lambda n}\sum^{0}_{k=-\infty}|E_{k}|^{\lambda}\left(\sum^{k-2}_{j=-\infty}2^{\frac{n(j-k)}{p'}}|E_{k}|^{\theta-\frac{\lambda}{q}}|E_{j}|^{\frac{\lambda}{q}-\theta}\right)^{q}\|f\|^{q}_{L^{p,q,\lambda}_{\alpha,\theta, V}(\mathbb{R}^{n})}\\
&\lesssim&(1+rm_{V}(x_{0}))^{\alpha}r^{-\lambda n}\sum^{0}_{k=-\infty}|E_{k}|^{\lambda}\left(\sum^{k-2}_{j=-\infty}2^{(j-k)n(1-\frac{1}{p}+\frac{\lambda}{q}-\theta)}\right)^{q}\|f\|^{q}_{L^{p,q,\lambda}_{\alpha,\theta, V}(\mathbb{R}^{n})}\\
&\lesssim&\|f\|^{q}_{L^{p,q,\lambda}_{\alpha,\theta, V}(\mathbb{R}^{n})}.
\end{eqnarray*}

For $A_{3}$, we can see that when $x\in E_{k}$, $y\in E_{j}$, then $|x-y|\sim2^{j}r$ for $j\geq k+2$. Similar to $E_{1}$, we have
\begin{eqnarray*}
\|\chi_{k}Tf_{j}\|_{L^{p}(\mathbb{R}^{n})}&\lesssim&\frac{1}{(1+2^{j}rm_{V}(x_{0}))^{N/k_{0}+1}}\frac{1}{(2^{j}r)^{n}}|E_{k}|^{\frac{1}{p}}\int_{E_{j}}|f(y)|dy\\
&\lesssim&\frac{1}{(1+2^{j}rm_{V}(x_{0}))^{N/k_{0}+1}}\frac{1}{(2^{j}r)^{n}}|E_{k}|^{\frac{1}{p}}|E_{j}|^{\frac{1}{p'}}\left(\int_{E_{j}}|f(y)|^{p}dy\right)^{\frac{1}{p}}\\
&\lesssim&\frac{1}{(1+2^{j}rm_{V}(x_{0}))^{N/k_{0}+1}}|E_{k}|^{\frac{1}{p}}|E_{j}|^{-\frac{1}{p}}\|\chi_{j}f\|_{L^{p}(\mathbb{R}^{n})}.
\end{eqnarray*}
Since $-\frac{1}{p}+\frac{\lambda}{q}<\theta<(1-\frac{1}{p})+\frac{\lambda}{q}$, choosing $N$ large enough, we obtain
\begin{eqnarray*}
A_{3}&\lesssim&(1+rm_{V}(x_{0}))^{\alpha}r^{-\lambda n}\sum^{0}_{k=-\infty}|E_{k}|^{\theta q}\left(\sum^{\infty}_{j=k+2}\frac{|E_{k}|^{\frac{1}{p}}|E_{j}|^{-\frac{1}{p}}\|\chi_{j}f\|_{L^{p}(\mathbb{R}^{n})}}{(1+2^{j}rm_{V}(x_{0}))^{N/k_{0}+1}}\right)^{q}\\
&\lesssim&(1+rm_{V}(x_{0}))^{\alpha}r^{-\lambda n}\sum^{0}_{k=-\infty}|E_{k}|^{\theta q}\Big\{\sum^{\infty}_{j=k+2}\frac{(1+2^{j}rm_{V}(x_{0}))^{-\frac{\alpha}{q}}(2^{j}r)^{\frac{\lambda n}{q}}|E_{j}|^{-\alpha}}{(1+2^{j}rm_{V}(x_{0}))^{N/k_{0}+1}}\\
&&\times 2^{(k-j)\frac{n}{p}}(1+2^{j}rm_{V}(x_{0}))^{\frac{\alpha}{q}}(2^{j}r)^{-\frac{\lambda n}{q}}(|E_{j}|^{\theta q}\|\chi_{j}f\|^{q}_{L^{p}(\mathbb{R}^{n})})^{\frac{1}{q}}\Big\}^{q}\\
&\lesssim&(1+rm_{V}(x_{0}))^{\alpha}r^{-\lambda n}\sum^{0}_{k=-\infty}|E_{k}|^{\theta q}\left(\sum^{\infty}_{j=k+2}2^{(k-j)\frac{n}{p}}|E_{j}|^{\frac{\lambda}{q}-\theta}\right)^{q}\|f\|^{q}_{L^{p,q,\lambda}_{\alpha,\theta, V}(\mathbb{R}^{n})}\\
&\lesssim&\|f\|^{q}_{L^{p,q,\lambda}_{\alpha,\theta, V}(\mathbb{R}^{n})}.
\end{eqnarray*}
Let $N=[-\frac{\alpha}{q}+1](k_{0}+1)$. Finally,
$\|Tf\|_{L^{p,q,\lambda}_{\alpha,\theta, V}(\mathbb{R}^{n})}\lesssim\|f\|_{L^{p,q,\lambda}_{\alpha,\theta, V}(\mathbb{R}^{n})}.$ This completes the proof of Theorem \ref{th-CZ}.
\end{proof}

Suppose that $b\in BMO(\mathbb{R}^{n})$ and $V\in B_{n}$.  Let $T$ be one of the Schrodinger type operators $\nabla(-\Delta+V)^{-1}\nabla$, $\nabla(-\Delta+V)^{-1/2}$ and $(-\Delta+V)^{-1/2}\nabla$. The commutator  $[b, T]$ is defined as
$$[b,T]f=bT(f)-T(bf).$$
\begin{theorem}\label{th-commutator}
Suppose that $V\in B_{n}$ and $b\in BMO(\mathbb{R}^{n})$. Let $1<p<\infty$, $1<q<\infty$, $\alpha\in (-\infty, 0], \lambda \in (0,n)$.
If the index $(p,q,\theta,\lambda)$ satisfies $\frac{\lambda}{q}-\frac{1}{p}<\theta<\frac{\lambda}{q}+1-\frac{1}{p}$, then
$$ \|[b,T]f\|_{L^{p,q,\lambda}_{\alpha,\theta,V}}\leq C\|f\|_{L^{p,q,\lambda}_{\alpha,\theta,V}}\|b\|_{BMO}.$$
\end{theorem}
\begin{proof}
For any ball $B=B(x_{0},r)$, we can get
$$f(y)=\sum^{\infty}_{j=-\infty}f(y)\chi_{E_{j}}(y)=\sum^{\infty}_{j=-\infty}f_{j}(y),$$
where $E_{j}=B(x_{0},2^{j}r)\backslash B(x_{0},2^{j-1}r)$. Hence, we have
\begin{eqnarray*}
&&(1+rm_{V}(x_{0}))^{\alpha}r^{-\lambda n}\sum^{0}_{k=-\infty}|E_{k}|^{\theta q}\|\chi_{k}[b,T]f\|^{q}_{L^{p}(\mathbb{R}^{n})} \\
&&\lesssim (1+rm_{V}(x_{0}))^{\alpha}r^{-\lambda n}\sum^{0}_{k=-\infty}|E_{k}|^{\theta q}\left(\sum^{k-2}_{j=-\infty}\|\chi_{k}[b,T]f_{j}\|_{L^{p}(\mathbb{R}^{n})}\right)^{q} \\
&&+(1+rm_{V}(x_{0}))^{\alpha}r^{-\lambda n}\sum^{0}_{k=-\infty}|E_{k}|^{\theta q}\left(\sum^{k+1}_{j=k-1}\|\chi_{k}[b,T]f_{j}\|_{L^{p}(\mathbb{R}^{n})}\right)^{q} \\
&&+(1+rm_{V}(x_{0}))^{\alpha}r^{-\lambda n}\sum^{0}_{k=-\infty}|E_{k}|^{\theta q}\left(\sum^{\infty}_{j=k+2}\|\chi_{k}[b,T]f_{j}\|_{L^{p}(\mathbb{R}^{n})}\right)^{q} \\
&&=: B_{1}+B_{2}+B_{3}.
\end{eqnarray*}
For $B_{2}$, by Corollary \ref{coro-commutator-bdd}, we have
\begin{eqnarray*}
B_{2}&\lesssim&(1+rm_{V}(x_{0}))^{\alpha}r^{-\lambda n}\sum^{0}_{k=-\infty}|E_{k}|^{\theta q}\left(\sum^{k+1}_{j=k-1}\|[b,T]f_{j}\|_{L^{p}(\mathbb{R}^{n})}\right)^{q}\\
&\lesssim& (1+rm_{V}(x_{0}))^{\alpha}r^{-\lambda n}\sum^{0}_{k=-\infty}|E_{k}|^{\theta q}\left(\sum^{k+1}_{j=k-1}\|f_{j}\|_{L^{p}(\mathbb{R}^{n})}\right)^{q}\|b\|^{q}_{BMO}\\
&\lesssim& \|f\|^{q}_{L^{p,q,\lambda}_{\alpha,\theta,V}}\|b\|^{q}_{BMO}.
\end{eqnarray*}
Denote by $b_{2^{k}r}$ the mean value of $b$ on the ball $B(x_{0}, 2^{k}r)$. For $B_{1}$, by Lemma \ref{le2.1} and (\ref{eq2.1}), we have
\begin{eqnarray*}
&&\|\chi_{k}[b,T]f_{j}\|_{L^{p}(\mathbb{R}^{n})}\\
&&\lesssim\ \frac{1}{(1+2^{k}rm_{V}(x_{0}))^{N/k_{0}+1}}\frac{1}{(2^{k}r)^{n}}\left[\int_{E_{k}}\Big(\int_{E_{j}}|b(x)-b(y)||f(y)|dy\Big)^{p}dx\right]^{\frac{1}{p}}\\
&&\lesssim\ \frac{1}{(1+2^{k}rm_{V}(x_{0}))^{N/k_{0}+1}}\frac{1}{(2^{k}r)^{n}}\Big[\Big(\int_{E_{k}}|b(x)-b_{2^{k}r}|^{p}dx\Big)^{\frac{1}{p}}\int_{E_{j}}|f(y)|dy\\
&&+|E_{k}|^{\frac{1}{p}}\int_{E_{j}}|b(y)-b_{2^{k}r}||f(y)|dy\Big]\\
&&\lesssim\ \frac{1}{(1+2^{k}rm_{V}(x_{0}))^{N/k_{0}+1}}\frac{1}{(2^{k}r)^{n}}\Big[|E_{k}|^{\frac{1}{p}}|E_{j}|^{1-\frac{1}{p}}\|b\|_{BMO}\|f_{j}\|_{L^{p}(\mathbb{R}^{n})}\\
&&\quad +|E_{k}|^{\frac{1}{p}}\|f_{j}\|_{L^{p}(\mathbb{R}^{n})}\Big(\int_{E_{j}}|b(y)-b_{2^{k}r}|^{p'}dx\Big)^{\frac{1}{p'}}\Big]\\
&&\lesssim
\frac{1}{(1+2^{k}rm_{V}(x_{0}))^{N/k_{0}+1}}\frac{|E_{j}|^{1-\frac{1}{p}}}{|E_{k}|^{1-\frac{1}{p}}}(k-j)\|f_{j}\|_{L^{p}(\mathbb{R}^{n})}\|b\|_{BMO},
\end{eqnarray*}
where in the third inequality, we have used John-Nirenberg's inequality (\cite{JN}).
Since $- \frac{1}{p}+ \frac{\lambda}{q} <\theta<(1- \frac{1}{p})+ \frac{\lambda}{q}$, we obtain
\begin{eqnarray*}
B_{1}&\lesssim& \frac{(1+rm_{V}(x_{0}))^{\alpha}}{r^{\lambda n}}\sum^{0}_{k=-\infty}|E_{k}|^{\theta q}\left(\sum^{k-2}_{j=-\infty}\frac{(k-j)\|f_{j}\|_{L^{p}(\mathbb{R}^{n})}}{(1+2^{k}rm_{V}(x_{0}))^{N/k_{0}+1}}
\frac{|E_{j}|^{1-\frac{1}{p}}}{|E_{k}|^{1-\frac{1}{p}}}\right)^{q}\|b\|^{q}_{BMO}\\
&\lesssim& \frac{(1+rm_{V}(x_{0}))^{\alpha}}{r^{\lambda n}}\sum^{0}_{k=-\infty}|E_{k}|^{\theta q}\Big[\sum^{k-2}_{j=-\infty}\frac{(1+2^{j}rm_{V}(x_{0}))^{-\frac{\alpha}{q}}}{(1+2^{k}rm_{V}(x_{0}))^{N/k_{0}+1}}\|b\|^{q}_{BMO}\\
&&\times(k-j)(2^{j}r)^{\frac{\lambda n}{q}}|E_{j}|^{-\theta}\frac{|E_{j}|^{1-\frac{1}{p}}}{|E_{k}|^{1-\frac{1}{p}}}\Big]^{q}\|f\|^{q}_{L_{\alpha,\theta ,V}^{p,q,\lambda}}\\
&\lesssim& \frac{(1+rm_{V}(x_{0}))^{\alpha}}{r^{\lambda n}}\sum^{0}_{k=-\infty}|E_{k}|^{\lambda}\left(\sum^{k-2}_{j=-\infty}(k-j)2^{(k-j)n(\theta-\frac{\alpha}{q}+\frac{1}{p}-1)}\right)^{q}
\|f\|^{q}_{L^{p,q,\lambda}_{\alpha,\theta,V}}\|b\|^{q}_{BMO}\\
&\lesssim& \|f\|^{q}_{L^{p,q,\lambda}_{\alpha,\theta,V}}\|b\|^{q}_{BMO}.
\end{eqnarray*}
For $B_{3}$, similar to $B_{1}$, we have
\begin{eqnarray*}
&&\|\chi_{k}[b,T]f_{j}\|_{L^{p}(\mathbb{R}^{n})}\\
&&\lesssim\ \frac{1}{(1+2^{j}rm_{V}(x_{0}))^{N/k_{0}+1}}\frac{1}{(2^{j}r)^{n}}\left(\int_{E_{k}}\mid\int_{E_{j}}|(b(x)-b(y))f(y)|dy\mid^{p}dx\right)^{\frac{1}{p}}\\
&&\lesssim\ \frac{j-k}{(1+2^{j}rm_{V}(x_{0}))^{N/k_{0}+1}}|E_{k}|^{\frac{1}{p}}|E_{j}|^{-\frac{1}{p}}\|f_{j}\|_{L^{p}(\mathbb{R}^{n})}\|b\|_{BMO}.
\end{eqnarray*}
Since $-\frac{1}{p}+\frac{\lambda}{q}<\theta<(1-\frac{1}{p})+\frac{\lambda}{q}$, choosing $N$ large enough, we obtain
\begin{eqnarray*}
B_{3}&\lesssim& (1+rm_{V}(x_{0}))^{\alpha}r^{-\lambda n}\sum^{0}_{k=-\infty}|E_{k}|^{\theta q}\left(\sum^{\infty}_{j=k+2}\frac{|E_{k}|^{\frac{1}{p}}|E_{j}|^{-\frac{1}{p}}(j-k)\|f_{j}\|_{L^{p}
(\mathbb{R}^{n})}}{(1+2^{j}rm_{V}(x_{0}))^{N/k_{0}+1}}\right)^{q}\|b\|^{q}_{BMO}\\
&\lesssim& (1+rm_{V}(x_{0}))^{\alpha}r^{-\lambda n}\sum^{0}_{k=-\infty}|E_{k}|^{\theta q}\Big[\sum^{\infty}_{j=k+2}\frac{(1+2^{j}rm_{V}(x_{0}))^{-\frac{\alpha}{q}}}{(1+2^{j}rm_{V}(x_{0}))^{N/k_{0}+1}}\\
&&\times (j-k)(2^{j}r)^{\frac{\lambda n}{q}}|E_{j}|^{-\theta}|E_{k}|^{\frac{1}{p}}|E_{j}|^{-\frac{1}{p}}\Big]^{q}\|f\|^{q}_{L^{p,q,\lambda}_{\alpha,\theta,V}(\mathbb{R}^{n})}\|b\|^{q}_{BMO}\\
&\lesssim& (1+rm_{V}(x_{0}))^{\alpha}r^{-\lambda n}\sum^{0}_{k=-\infty}|E_{k}|^{\lambda}\left(\sum^{\infty}_{j=k+2}2^{(k-j)n(\frac{1}{p}-\frac{\lambda}{q}+\theta)}\right)^{q}
\|f\|^{q}_{L^{p,q,\lambda}_{\alpha,\theta,V}(\mathbb{R}^{n})}\|b\|^{q}_{BMO}\\
&\lesssim& \|f\|^{q}_{L^{p,q,\lambda}_{\alpha,\theta,V}(\mathbb{R}^{n})}\|b\|^{q}_{BMO}.
\end{eqnarray*}
Let $N=[-\frac{\alpha}{q}+1](k_{0}+1)$. We finally get
$$\|[b,T]f\|_{L^{p,q,\lambda}_{\alpha,\theta,V}(\mathbb{R}^{n})}\lesssim \|f\|_{L^{p,q,\lambda}_{\alpha,\theta,V}(\mathbb{R}^{n})}\|b\|_{BMO}.$$
\end{proof}

\section{Schr\"odinger type operators on $L^{p,q,\lambda}_{\alpha,\theta,V}(\mathbb{R}^{n})$}\label{sec3}
Let $L=-\Delta+V$ be the Schr\"odinger operator, where $V\in B_{s}, s>n/2$. For $0<\beta<\frac{n}{2}$, the fractional integral operator associated with $L$ is defined by
$$L^{-\beta}(f)(x)=\int^{\infty}_{0}e^{-tL}(f)(x)t^{\beta-1}dt.$$
Denote by $K_{\beta}(x,y)$ the kernel  of $L^{-\beta}$.
By Lemma \ref{le3.1}, Bui \cite{Bui} obtained the following pointwise estimate.
\begin{lemma}\label{le3.2}{\rm (\cite[Proposition 3.3]{Bui})}
Let $0<\beta<\frac{n}{2}$. For  $N\in\mathbb{N}$, there is a constant $C_{N}$ such that
\begin{equation}\label{eq-Kbeta}
 K_{\beta}(x,y)=\int_{0}^{\infty}K_{t}(x,y)t^{\beta-1}dt\leq \frac{C_{N}}{(1+\mid x-y\mid m_{V}(x))^{N}}\frac{1}{|x-y|^{n-2\beta}},
 \end{equation}
 where $K_{t}(\cdot, \cdot)$ is the kernel of the semigroup $e^{-tL}$.
\end{lemma}

\begin{definition}
Let $f\in L_{loc}^{q}(\mathbb{R}^{n})$. Denote by $|B|$  the Lebesgue measure of the ball $B\subset \mathbb{R}^{n}$. The fractional Hardy-Littlewood  maximal function $M_{\sigma,\gamma}$ is defined by
$$M_{\sigma,\gamma}f(x)=\sup_{x\in B}\left(\frac{1}{|B|^{1-\frac{\sigma \gamma}{n}}}\int_{B}|f(y)|^{\gamma}dy\right)^{\frac{1}{\gamma}}.$$
\end{definition}

\begin{lemma}\label{le3.5}
{\rm (\cite{CRW})} Suppose  $1<\gamma<p_{1}<\frac{n}{\sigma}$ and $\frac{1}{p_{2}}=\frac{1}{p_{1}}-\frac{\sigma}{n}$.  Then
$$\|M_{\sigma,\gamma}f\|_{L^{p_{2}}(\mathbb{R}^{n})}\lesssim\|f\|_{L^{p_{1}}(\mathbb{R}^{n})}.$$
\end{lemma}

As a generalization of the fractional integral associated with $L$, the operators $V^{\beta_{2}}(-\Delta+V)^{-\beta_{1}}, 0\leq\beta_{2}\leq\beta_{1}\leq1$, have been studied by Sugano \cite{Su} systematically. Applying the method of Sugano \cite{Su} together with Lemma \ref{le3.2}, we can obtain the following result for $V^{\beta_{2}}(-\Delta+V)^{-\beta_{1}}, 0\leq\beta_{2}\leq\beta_{1}\leq n/2$. We omit the proof.
\begin{theorem}
 Suppose that $V\in B_{\infty}$. Let $1<\beta_{2}\leq\beta_{1}<\frac{n}{2}$. Then
$$|V^{\beta_{2}}(-\Delta+V)^{-\beta_{1}}f(x)| \lesssim  M_{2(\beta_{1}-\beta_{2}),1}f(x).$$
\end{theorem}
In a similar way,  by (\ref{eq-Kbeta}), we can get the following estimate for the operators $(-\Delta+V)^{-\beta_{1}}V^{\beta_{2}}$, $0\leq\beta_{2}\leq\beta_{1}<\frac{n}{2}$.
\begin{theorem}\label{th-4.5}
Suppose that $V\in B_{s}$ for $s>\frac{n}{2}$. Let $0\leq\beta_{2}\leq\beta_{1}<\frac{n}{2}$. Then
$$|(-\Delta+V)^{-\beta_{1}}(V^{\beta_{2}}f)(x)| \lesssim  M_{2(\beta_{1}-\beta_{2})},$$
where $(\frac{s}{\beta_{2}})'$ is the conjugate of $(\frac{s}{\beta_{2}})$.
\end{theorem}

\begin{proof}
Let $r={1}/{m_{V}(x)}$. By Lemma \ref{le3.2} and H\"older's inequality, we have
\begin{eqnarray*}
&&|(-\Delta+V)^{-\beta_{1}}V^{\beta_{2}}(x)f(x)|\\
&&\lesssim\sum^{\infty}_{k=-\infty}\int_{2^{k-1}r \leq|x-y|\leq 2^{k}r}\frac{1}{(1+2^{k}rm_{V}(x_{0}))^{N}}\frac{1}{(2^{k}r)^{n-2{\beta_{1}}}}V(y)^{\beta_{2}}|f(y)|dy\\
&&\lesssim \sum^{\infty}_{k=-\infty}\frac{(2^{k}r)^{2\beta_{2}}}{(1+2^{k})^{N}}
\left(\frac{1}{(2^{k}r)^{n}}\int_{B(x,2^{k}r)}V(y)dy\right)^{\beta_{2}}M_{2(\beta_{1}-\beta_{2}),(\frac{s}{\beta_{2}})'}(f)(x).\\
\end{eqnarray*}
For $k\geq 1$, because $V(y)dy$ is a doubling measure, we have
\begin{eqnarray*}
\frac{(2^{k}r)^{2}}{(2^{k}r)^{n}}\int_{B(x,2^{k}r)}V(y)dy &\lesssim& C_{0}^{k}\cdot 2^{(2-n)k} \frac{r^{2}}{r^{n}}\int_{B(x,r)}V(y)dy\\
&\lesssim&(2^{k})^{k_{0}},
\end{eqnarray*}
where $k_{0}=2-n+\log_{2}C_{0}$. For $k\leq 0$, Lemma \ref{le3.3.0} implies that
\begin{eqnarray*}
\frac{(2^{k}r)^{2}}{(2^{k}r)^{n}}\int_{B(x,2^{k}r)}V(y)dy&\lesssim&\Big(\frac{r}{2^{k}r}\Big)^{\frac{n}{s}-2}\frac{r^{2}}{r^{n}}\int_{B(x,r)}V(y)dy\\
&\lesssim&(2^{k})^{2-\frac{n}{s}}.
\end{eqnarray*}
Taking $N$ large enough, we get
$$|(-\Delta+V)^{-\beta_{1}}V^{\beta_{2}}f(x)| \lesssim M_{2(\beta_{1}-\beta_{2}),(\frac{s}{\beta_{2}})^{'}}f(x).$$
\end{proof}
By Theorem \ref{th-4.5} and the duality, we can obtain
\begin{corollary}\label{coro3.6}
Suppose $V\in B_{s}$ for $s>\frac{n}{2}$.
\begin{itemize}
\item[(1)] If $1<(\frac{s}{\beta_{2}})'<p_{1}<\frac{n}{2\beta_{1}-2\beta_{2}}$ and $\frac{1}{p_{2}}=\frac{1}{p_{1}}-\frac{2\beta_{1}-2\beta_{2}}{n}$, then
$$\|(-\Delta+V)^{-\beta_{1}}V^{\beta_{2}}f\|_{L^{p_{2}}(\mathbb{R}^{n})}\lesssim\|f\|_{L^{p_{1}}(\mathbb{R}^{n})},$$
where $\frac{s}{\beta_{2}}+(\frac{s}{\beta_{2}})'=1$.
\item[(2)] If $1<p_{2}<\frac{s}{\beta_{2}}$ and $\frac{1}{p_{2}}=\frac{1}{p_{1}}-\frac{2\beta_{1}-2\beta_{2}}{n}$, then
$$\|V^{\beta_{2}}(-\Delta+V)^{-\beta_{1}}f\|_{L^{p_{2}}(\mathbb{R}^{n})}\lesssim\|f\|_{L^{p_{1}}(\mathbb{R}^{n})}.$$
\end{itemize}
\end{corollary}

\begin{theorem}\label{th-operator}
Suppose that $V\in B_{s}$, $s\geq\frac{n}{2}$, $\alpha\in(-\infty,0]$, $\lambda\in(0,n)$. Let $1<q<\infty$,
 $1<\beta_{2}\leq\beta_{1}<\frac{n}{2}$ and  $1<p_{2}<\frac{s}{\beta_{2}}$ with $\frac{1}{p_{1}}-\frac{1}{p_{2}}=\frac{2\beta_{1}-2\beta_{2}}{n}$. If $\frac{\lambda}{q}-\frac{1}{p_{1}}+\frac{2\beta_{1}}{n}<\theta<\frac{\lambda}{q}+1-\frac{1}{p_{1}}$, then
$$ \|V^{\beta_{2}}(-\Delta+V)^{-\beta_{1}}f\|_{L^{p_{2},q,\lambda}_{\alpha,\theta, V}}\lesssim\|f\|_{L^{p_{1},q,\lambda}_{\alpha,\theta, V}}.$$
\end{theorem}
\begin{proof}
For any ball $B(x_{0},r)$, write
$$f(y)=\sum^{\infty}_{j=-\infty}f(y)\chi_{E_{j}}(y)=\sum^{\infty}_{j=-\infty}f_{j}(y),$$
where $E_{j}=B(x_{0},2^{j}r)\backslash B(x_{0},2^{j-1}r)$. Hence, we have
\begin{eqnarray*}
&&(1+rm_{V}(x_{0}))^{\alpha}r^{-\lambda n}\sum^{0}_{k=-\infty}|E_{k}|^{\theta q}\|\chi_{k}V^{\beta_{2}}(-\Delta+V)^{-\beta_{1}}f\|^{q}_{L^{p_{2}}(\mathbb{R}^{n})}\\
&&\lesssim(1+rm_{V}(x_{0}))^{\alpha}r^{-\lambda n}\sum^{0}_{k=-\infty}|E_{k}|^{\theta q}\left(\sum^{k-2}_{j=-\infty}\|\chi_{k}V^{\beta_{2}}(-\Delta+V)^{-\beta_{1}}f_{j}\|_{L^{p_{2}}(\mathbb{R}^{n})}\right)^{q}\\
&&+(1+rm_{V}(x_{0}))^{\alpha}r^{-\lambda n}\sum^{0}_{k=-\infty}|E_{k}|^{\theta q}\left(\sum^{k+1}_{j=k-1}\|\chi_{k}V^{\beta_{2}}(-\Delta+V)^{-\beta_{1}}f_{j}\|_{L^{p_{2}}(\mathbb{R}^{n})}\right)^{q}\\
&&+(1+rm_{V}(x_{0}))^{\alpha}r^{-\lambda n}\sum^{0}_{k=-\infty}|E_{k}|^{\theta q}\left(\sum^{\infty}_{j=k+2}\|\chi_{k}V^{\beta_{2}}(-\Delta+V)^{-\beta_{1}}f_{j}\|_{L^{p_{2}}(\mathbb{R}^{n})}\right)^{q}\\
&&=M_{1}+M_{2}+M_{3}.
\end{eqnarray*}

We first estimate $M_{2}$.  For $1<p_{2}<\frac{s}{\beta_{2}}$, by (2) of Corollary \ref{coro3.6},  we can get
\begin{eqnarray*}
M_{2}&\lesssim&\frac{(1+rm_{V}(x_{0}))^{\alpha}}{r^{\lambda n}}\sum^{0}_{k=-\infty}|E_{k}|^{\theta q}\left(\sum^{k+1}_{j=k-1}\|f_{j}\|_{L^{p_{1}}(\mathbb{R}^{n})}\right)^{q}
\lesssim\|f\|^{q}_{L^{p_{1},q,\lambda}_{\alpha,\theta,V}}.
\end{eqnarray*}

Now we deal with the terms $M_{1}$ and $M_{3}$. We choose $N$ large enough such that
$$(N/k_{0}+1)-(\log_{2}C_{0}+1)\beta_{2}+{\alpha}/{q}>0$$ and take a positive $N_{1}<(N/k_{0}+1)-(\log_{2}C_{0}+1)\beta_{2}$. For $M_{1}$, note that if $x\in E_{k}$, $y\in E_{j}$ and $j\leq k-2$, then $|x-y|\sim 2^{k}r$. By Lemmas \ref{le3.2} \& \ref{le3.3}, we use H\"older's inequality to obtain
\begin{eqnarray*}
&&\|\chi_{k}V^{\beta_{2}}(-\Delta+V)^{-\beta_{1}}f_{j}\|_{L^{p_{2}}(\mathbb{R}^{n})}\\
&&\lesssim \left(\int_{E_{k}}\mid V^{\beta_{2}}(x)\int_{E_{j}}\frac{1}{(1+\mid x-y\mid m_{v}(x))^{N}}\frac{1}{|x-y|^{n-2{\beta_{1}}}}f(y)dy\mid^{p_{2}}dx\right)^{\frac{1}{p_{2}}}\\
&&\lesssim\frac{1}{(1+2^{k}rm_{V}(x_{0}))^{N/k_{0}+1}}\frac{1}{(2^{k}r)^{n-2{\beta_{1}}}}\int_{E_{j}}|f(y)|dy\left(\int_{E_{k}}|V(x)|^{\beta_{2} p_{2}}dx\right)^{\frac{1}{p_{2}}}\\
&&\lesssim\frac{|E_{j}|^{1-\frac{1}{p_{1}}}|E_{k}|^{\frac{1}{p_{2}}}}{(1+2^{k}rm_{V}(x_{0}))^{N/k_{0}+1}}\frac{1}{(2^{k}r)^{n-2{\beta_{1}}}}\|f_{j}\|_{L^{p_{1}}(\mathbb{R}^{n})}
\left(\frac{1}{|E_{k}|}\int_{E_{k}}V(x)^{s}dx\right)^{\frac{\beta_{2}}{s}}\\
&&\lesssim\frac{|E_{j}|^{1-\frac{1}{p_{1}}}|E_{k}|^{\frac{1}{p_{2}}}}{(1+2^{k}rm_{V}(x_{0}))^{N/k_{0}+1}}\frac{1}{(2^{k}r)^{n-2{\beta_{1}}}}\|f_{j}\|_{L^{p_{1}}(\mathbb{R}^{n})}
\left(\frac{1}{|B_{k}|}\int_{B_{k}}V(x)dx\right)^{\beta_{2}}\\
&&\lesssim\frac{1}{(1+2^{k}rm_{V}(x_{0}))^{N_{1}}}\frac{1}{(2^{k}r)^{n-2\beta_{1}+2\beta_{2}}}|E_{k}|^{\frac{1}{p_{2}}}|E_{j}|^{1-\frac{1}{p_{1}}}\|f_{j}\|_{L^{p_{1}}(\mathbb{R}^{n})},
\end{eqnarray*}
where $\frac{1}{p_{1}}-\frac{1}{p_{2}}=\frac{2\beta_{1}-2\beta_{2}}{n}$. Since $\frac{\lambda}{q}-\frac{1}{p_{1}}+\frac{2\beta_{1}}{n}<\theta<\frac{\lambda}{q}+1-\frac{1}{p_{1}}$, we obtain
\begin{eqnarray*}
M_{1}&\lesssim&(1+rm_{V}(x_{0}))^{\alpha}r^{-\lambda n}\sum^{0}_{k=-\infty}|E_{k}|^{\theta q}\\
&&\times\left(\sum^{k-2}_{j=-\infty}\frac{1}{(1+2^{k}rm_{V}(x_{0}))^{N_{1}}}\frac{1}{(2^{k}r)^{n-2\beta_{1}+2\beta_{2}}}|E_{k}|^{\frac{1}{p_{2}}}|E_{j}|^{1-\frac{1}{p_{1}}}\|f_{j}\|_{L^{p_{1}}(\mathbb{R}^{n})}\right)^{q}\\
&\lesssim&(1+rm_{V}(x_{0}))^{\alpha}r^{-\lambda n}\sum^{0}_{k=-\infty}|E_{k}|^{\theta q}\\
&&\times\left(\sum^{k-2}_{j=-\infty}\frac{(1+2^{j}rm_{V}(x_{0}))^{-\frac{\alpha}{q}}}{(1+2^{k}rm_{V}(x_{0}))^{N_{1}}}\frac{(2^{j}r)^{\frac{\lambda n}{q}}|E_{j}|^{-\theta}}{(2^{k}r)^{n-2\beta_{1}+2\beta_{2}}}|E_{k}|^{\frac{1}{p_{2}}}|E_{j}|^{1-\frac{1}{p_{1}}}\right)^{q}
\|f\|^{q}_{L^{p_{1},\lambda,q}_{\alpha,v,\theta}}\\
&\lesssim&(1+rm_{V}(x_{0}))^{\alpha}r^{-\lambda n}\sum^{0}_{k=-\infty}|E_{k}|^{\lambda}\left(\sum^{k-2}_{j=-\infty}2^{(j-k)n(\frac{\lambda}{q}-\theta-\frac{1}{p_{1}}+1)}\right)^{q}
\|f\|^{q}_{L^{p_{1},\lambda,q}_{\alpha,v,\theta}}\\
&\lesssim&\|f\|^{q}_{L^{p_{1},\lambda,q}_{\alpha,v,\theta}}.
\end{eqnarray*}

For $M_{3}$,  note that when $x\in E_{k}$, $y\in E_{j}$ and $j\geq k+2$, then $|x-y|\sim2^{j}r$.  Similar to $E_{1}$, we have
\begin{eqnarray*}
&&\|\chi_{k}V^{\beta_{2}}(-\Delta+V)^{-\beta_{1}}f_{j}\|_{L^{p_{2}}(\mathbb{R}^{n})}\\
&&\lesssim\frac{1}{(1+2^{j}rm_{V}(x_{0}))^{N/k_{0}+1}}\frac{1}{(2^{j}r)^{n-2{\beta_{1}}}}\int_{E_{j}}|f(y)|dy\left(\int_{E_{k}}|V(x)|^{\beta_{2} p_{2}}dx\right)^{\frac{1}{p_{2}}}\\
&&\lesssim\frac{1}{(1+2^{j}rm_{V}(x_{0}))^{N_{1}}}|E_{j}|^{\frac{2\beta_{1}}{n}-\frac{1}{p_{1}}}|E_{k}|^{\frac{1}{p_{2}}-\frac{2\beta_{2}}{n}}\|f_{j}\|_{L^{p_{1}}(\mathbb{R}^{n})},
\end{eqnarray*}
where $\frac{1}{p_{1}}-\frac{1}{p_{2}}=\frac{2\beta_{1}-2\beta_{2}}{n}$. Since $\frac{\lambda}{q}-\frac{1}{p_{1}}+\frac{2\beta_{1}}{n}<\theta<\frac{\lambda}{q}+1-\frac{1}{p_{1}}$, we obtain
\begin{eqnarray*}
M_{3}&\lesssim&(1+rm_{V}(x_{0}))^{\alpha}r^{-\lambda n}\sum^{0}_{k=-\infty}|E_{k}|^{\theta q}\\
&&\times\left(\sum^{\infty}_{j=k+2}\frac{1}{(1+2^{j}rm_{V}(x_{0}))^{N_{1}}}|E_{j}|^{\frac{2\beta_{1}}{n}-\frac{1}{p_{1}}}|E_{k}|^{\frac{1}{p_{2}}-\frac{2\beta_{2}}{n}}\|f_{j}\|_{L^{p_{1}}(\mathbb{R}^{n})}\right)^{q}\\
&\lesssim&(1+rm_{V}(x_{0}))^{\alpha}r^{-\lambda n}\sum^{0}_{k=-\infty}|E_{k}|^{\theta q}\\
&&\times\left(\sum^{\infty}_{j=k+2}\frac{(1+2^{j}rm_{V}(x_{0}))^{-\frac{\alpha}{q}}(2^{j}r)^{\frac{\lambda n}{q}}|E_{j}|^{-\theta}}{(1+2^{j}rm_{V}(x_{0}))^{N_{1}}}
\frac{|E_{k}|^{\frac{1}{p_{2}}-\frac{2\beta_{2}}{n}}}{|E_{j}|^{\frac{2\beta_{1}}{n}-\frac{1}{p_{1}}}}\right)^{q}
\|f\|^{q}_{L^{p_{1},q,\lambda}_{\alpha,\theta,V}}\\
&\lesssim&(1+rm_{V}(x_{0}))^{\alpha}r^{-\lambda n}\sum^{0}_{k=-\infty}|E_{k}|^{\lambda}\left(\sum^{\infty}_{j=k+2}2^{(k-j)n(\theta-\frac{\lambda}{q}+\frac{1}{p_{1}}+\frac{2\beta_{1}}{n})}\right)^{q}
\|f\|^{q}_{L^{p_{1},q,\lambda}_{\alpha,\theta,V}}\\
&\lesssim&\|f\|^{q}_{L^{p_{1},q,\lambda}_{\alpha,\theta,V}}.
\end{eqnarray*}
Choosing $N$ large enough, we obtain
$$\|V^{\beta_{2}}(-\Delta+V)^{-\beta_{1}}f\|_{L^{p_{2},q,\lambda}_{\alpha,\theta,V}}
\lesssim\|f\|_{L^{p_{1},q,\lambda}_{\alpha,\theta,V}}.$$
\end{proof}

\begin{theorem}\label{th-dual-operator}
Suppose that $V\in B_{s}$, $s\geq\frac{n}{2}$, $\alpha\in (-\infty,0]$, $\lambda\in (0,n)$ and $1<q<\infty.$
Let $0<\beta_{2}\leq\beta_{1}<\frac{n}{2}$, $\frac{s}{s-\beta_{2}}<p_{1}<\frac{n}{2\beta_{1}-2\beta_{2}}$ with $\frac{1}{p_{2}}=\frac{1}{p_{1}}-\frac{2\beta_{1}-2\beta_{2}}{n}$. If $\frac{\lambda}{q}-\frac{1}{p_{2}}<\theta<\frac{\lambda}{q}-\frac{1}{p_{2}}+1-\frac{2\beta_{1}}{n}$, then
$$ \|(-\Delta+V)^{-\beta_{1}}V^{\beta_{2}}f\|_{L^{p_{2},q,\lambda}_{\alpha,\theta,V}}
\lesssim\|f\|_{L^{p_{1},q,\lambda}_{\alpha,\theta,V}}.$$
\end{theorem}
\begin{proof}
For any ball $B(x_{0}, r)$, let $E_{j}=B(x_{0},2^{j}r)\backslash B(x_{0},2^{j-1}r)$. We can decompose $f$ as follows.
$$f(y)=\sum^{\infty}_{j=-\infty}f(y)\chi_{E_{j}}(y)=\sum^{\infty}_{j=-\infty}f_{j}(y).$$
Similar to the proof of Theorem \ref{th-operator}, we have
\begin{eqnarray*}
&&(1+rm_{V}(x_{0}))^{\alpha}r^{-\lambda n}\sum^{0}_{k=-\infty}|E_{k}|^{\theta q}\|\chi_{k}(-\Delta+V)^{-\beta_{1}}V^{\beta_{2}}f\|^{q}_{L^{p_{2}}(\mathbb{R}^{n})}\\
&&\lesssim(1+rm_{V}(x_{0}))^{\alpha}r^{-\lambda n}\sum^{0}_{k=-\infty}|E_{k}|^{\theta q}\left(\sum^{k-2}_{j=-\infty}\|\chi_{k}(-\Delta+V)^{-\beta_{1}}V^{\beta_{2}}f_{j}\|_{L^{p_{2}}(\mathbb{R}^{n})}\right)^{q}\\
&&+C(1+rm_{V}(x_{0}))^{\alpha}r^{-\lambda n}\sum^{0}_{k=-\infty}|E_{k}|^{\theta q}\left(\sum^{k+1}_{j=k-1}\|\chi_{k}(-\Delta+V)^{-\beta_{1}}V^{\beta_{2}}f_{j}\|_{L^{p_{2}}(\mathbb{R}^{n})}\right)^{q}\\
&&+C(1+rm_{V}(x_{0}))^{\alpha}r^{-\lambda n}\sum^{0}_{k=-\infty}|E_{k}|^{\theta q}\left(\sum^{\infty}_{j=k+2}\|\chi_{k}(-\Delta+V)^{-\beta_{1}}V^{\beta_{2}}f_{j}\|_{L^{p_{2}}(\mathbb{R}^{n})}\right)^{q}\\
&&=L_{1}+L_{2}+L_{3}.
\end{eqnarray*}

For $L_{2}$ , because $1<\frac{s}{s-\beta_{2}}<p_{1}<\frac{n}{2\beta_{1}-\beta_{2}}$, we use Corollary \ref{coro3.6} to obtain
\begin{eqnarray*}
L_{2}&\lesssim&\frac{(1+rm_{V}(x_{0}))^{\alpha}}{r^{\lambda n}}\sum^{0}_{k=-\infty}|E_{k}|^{\theta q}\left(\sum^{k+1}_{j=k-1}\|f_{j}\|_{L^{p_{1}}(\mathbb{R}^{n})}\right)^{q}
\lesssim\|f\|^{q}_{L^{p_{1},q,\lambda}_{\alpha,\theta,V}}.
\end{eqnarray*}

For $L_{1}$,  we can see that if $x\in E_{k}$ and $y\in E_{j}$,  then $|x-y|\sim 2^{k}r$ for $j\leq k-2$. By H\"older's inequality and the fact that $V\in B_{s}$, we deduce from Lemmas \ref{le3.2} \& \ref{le3.3} that
\begin{eqnarray*}
&&\|\chi_{k}(-\Delta+V)^{-\beta_{1}}V^{\beta_{2}}f_{j}\|^{q}_{L^{p_{2}}(\mathbb{R}^{n})}\\
&&\lesssim\frac{1}{(1+2^{k}rm_{V}(x_{0}))^{N/k_{0}+1}}\frac{|E_{k}|^{\frac{1}{p_{2}}}}{(2^{k}r)^{n-2{\beta_{1}}}}\int_{E_{j}}V(x)^{\beta_{2}}|f(y)|dy\\
&&\lesssim\frac{1}{(1+2^{k}rm_{V}(x_{0}))^{N/k_{0}+1}}\frac{|E_{k}|^{\frac{1}{p_{2}}}}{(2^{k}r)^{n-2{\beta_{1}}}}|E_{j}|^{1-\frac{1}{p_{1}}}
\left(\frac{1}{|B_{j}|}\int_{B_{j}}V(x)dx\right)^{\beta_{2}}\|f_{j}\|_{L^{p_{1}}(\mathbb{R}^{n})}\\
&&\lesssim\frac{1}{(1+2^{k}rm_{V}(x_{0}))^{N_{2}}}\frac{|E_{k}|^{\frac{1}{p_{2}}}}{(2^{k}r)^{n-2{\beta_{1}}}}
|E_{j}|^{1-\frac{1}{p_{1}}}(2^{j}r)^{-2\beta_{2}}\|f_{j}\|_{L^{p_{1}}(\mathbb{R}^{n})},
\end{eqnarray*}
where $\frac{1}{p_{2}}=\frac{1}{p_{1}}-\frac{2\beta_{1}-2\beta_{2}}{n}$ and $N_{2}<(N/k_{0}+1)-(\log_{2}C_{0}+1)\beta_{2}$. Since $\frac{\lambda}{q}-\frac{1}{p_{2}}<\theta<\frac{\lambda}{q}-\frac{1}{p_{2}}+1-\frac{2\beta_{1}}{n}$, we obtain
\begin{eqnarray*}
L_{1}&\lesssim&(1+rm_{V}(x_{0}))^{\alpha}r^{-\lambda n}\sum^{0}_{k=-\infty}|E_{k}|^{\theta q}\\
&&\times\left(\sum^{k-2}_{j=-\infty}\frac{1}{(1+2^{k}rm_{V}(x_{0}))^{N_{2}}}
\frac{|E_{k}|^{\frac{1}{p_{2}}}}{(2^{k}r)^{n-2{\beta_{1}}}}|E_{j}|^{1-\frac{1}{p_{1}}}(2^{j}r)^{-2\beta_{2}}\|f_{j}\|_{L^{p_{1}}(\mathbb{R}^{n})}\right)^{q}\\
&\lesssim&(1+rm_{V}(x_{0}))^{\alpha}r^{-\lambda n}\sum^{0}_{k=-\infty}|E_{k}|^{\theta q}\\
&&\left(\sum^{k-2}_{j=-\infty}\frac{(1+2^{j}rm_{V}(x_{0}))^{-\frac{\alpha}{q}}}{(1+2^{k}rm_{V}(x_{0}))^{N_{2}}}
\frac{(2^{j}r)^{\frac{\lambda n}{q}}|E_{j}|^{-\theta}}{(2^{k}r)^{n-2{\beta_{1}}}}\frac{|E_{k}|^{\frac{1}{p_{2}}}|E_{j}|^{1-\frac{1}{p_{1}}}}{(2^{j}r)^{2\beta_{2}}}
\right)^{q}\|f\|^{q}_{L^{p_{1},\lambda,q}_{\alpha,V,\theta}}\\
&\lesssim&(1+rm_{V}(x_{0}))^{\alpha}r^{-\lambda n}\sum^{0}_{k=-\infty}|E_{k}|^{\lambda}\left(\sum^{k-2}_{j=-\infty}2^{(k-j)n(\theta-\frac{\lambda}{q}+\frac{1}{p_{2}}-1+\frac{2\beta_{1}}{n})}\right)^{q}
\|f\|^{q}_{L^{p_{1},\lambda,q}_{\alpha,\theta,V}}\\
&\lesssim&\|f\|^{q}_{L^{p_{1},q,\lambda}_{\alpha, V,\theta}}.
\end{eqnarray*}

For $L_{3}$,  note that when $x\in E_{k}$, $y\in E_{j}$ and $j\geq k+2$, then $|x-y|\sim2^{j}r$. Similar to $E_{1}$, we have
\begin{eqnarray*}
&&\|\chi_{k}(-\Delta+V)^{-\beta_{1}}V^{\beta_{2}}f_{j}\|^{q}_{L^{p_{2}}(\mathbb{R}^{n})}\\
&&\lesssim\frac{1}{(1+2^{j}rm_{V}(x_{0}))^{N/k_{0}+1}}\frac{|E_{k}|^{\frac{1}{p_{2}}}}{(2^{j}r)^{n-2{\beta_{1}}}}\int_{E_{j}}V(x)^{\beta_{2}}|f(y)|dy\\
&&\lesssim\frac{1}{(1+2^{j}rm_{V}(x_{0}))^{N_{2}}}\frac{|E_{k}|^{\frac{1}{p_{2}}}}{(2^{j}r)^{n-2{\beta_{1}}}}|E_{j}|^{1-\frac{1}{p_{1}}}(2^{j}r)^{-2\beta_{2}}
\|f_{j}\|_{L^{p_{1}}(\mathbb{R}^{n})},
\end{eqnarray*}
where $\frac{1}{p_{2}}=\frac{1}{p_{1}}-\frac{2\beta_{1}-2\beta_{2}}{n}$, and $N_{2}<(N/k_{0}+1)-(\log_{2}C_{0}+1)\beta_{2}$.
Since $\frac{\lambda}{q}-\frac{1}{p_{2}}<\theta<\frac{\lambda}{q}-\frac{1}{p_{2}}+1-\frac{2\beta_{1}}{n}$, we obtain
\begin{eqnarray*}
L_{3}&\lesssim&(1+rm_{V}(x_{0}))^{\alpha}r^{-\lambda n}\sum^{0}_{k=-\infty}|E_{k}|^{\theta q}\\
&&\times\left(\sum^{\infty}_{j=k+2}\frac{1}{(1+2^{j}rm_{V}(x_{0}))^{N_{2}}}\frac{|E_{k}|^{\frac{1}{p_{2}}}}{(2^{j}r)^{n-2{\beta_{1}}}}|E_{j}|^{1-\frac{1}{p_{1}}}(2^{j}r)^{-2\beta_{2}}\|f_{j}\|_{L^{p_{1}}(\mathbb{R}^{n})}\right)^{q}\\
&\lesssim&(1+rm_{V}(x_{0}))^{\alpha}r^{-\lambda n}\sum^{0}_{k=-\infty}|E_{k}|^{\lambda}\left(\sum^{\infty}_{j=k+2}2^{(k-j)n(\theta-\frac{\lambda}{q}+\frac{1}{p_{2}})}\right)^{q}
\|f\|^{q}_{L^{p_{1},q,\lambda}_{\alpha,\theta,V}}\\
&\lesssim&\|f\|^{q}_{L^{p_{1},q,\lambda}_{\alpha,\theta,V}}.
\end{eqnarray*}
Let $N$ large enough. We finally get
$\|(-\Delta+V)^{-\beta_{1}}V^{\beta_{2}}f\|_{L^{p_{2},q,\lambda}_{\alpha,\theta,V}}\lesssim\|f\|_{L^{p_{1},q,\lambda}_{\alpha,\theta,V}}$.
\end{proof}

\section{Boundedness of the commutators on $L^{p,q,\lambda}_{\alpha,\theta,V}(\mathbb{R}^{n})$}\label{sec-5}
In this section, let $b\in BMO(\mathbb{R}^{n})$. We consider the boundedness of commutators $[b, (-\Delta+V)^{-\beta_{1}}V^{\beta_{2}}]$ and its duality  on the generalized Morrey spaces $L^{p,q,\lambda}_{\alpha,\theta,V}(\mathbb{R}^{n})$. For this purpose, we prove the commutator $[b,\  (-\Delta)^{-\beta_{1}}V^{\beta_{2}}]$ is bounded from  $L^{p_{1}}(\mathbb{R}^{n})$ to $L^{p_{2}}(\mathbb{R}^{n})$. For sake of simplicity,  we denote by  $b_{2^{k}r}$ the mean value of $b$ on the ball $B(x_{0}, 2^{k}r)$.
\begin{theorem}\label{th-bdd-Tbeta}
Suppose that $V\in B_{s}$, $s\geq\frac{n}{2}$ and $b\in BMO(\mathbb{R}^{n})$.
\begin{itemize}
\item[(i)] If $0<\beta_{2}\leq\beta_{1}<\frac{n}{2}$ , $\frac{s}{s-\beta_{2}}<p_{1}<\frac{n}{2\beta_{1}-2\beta_{2}}$ , $\frac{1}{p_{2}}=\frac{1}{p_{1}}-\frac{2\beta_{1}-2\beta_{2}}{n}$, then
$$ \Big\|[b,\ (-\Delta+V)^{-\beta_{1}}V^{\beta_{2}}]f\Big\|_{L^{p_{2}}(\mathbb{R}^{n})}\lesssim \|f\|_{L^{p_{1}}(\mathbb{R}^{n})}\|b\|_{BMO}.$$
\item[(ii)] If $1<p_{2}<\frac{s}{\beta_{2}}$ and $\frac{1}{p_{2}}=\frac{1}{p_{1}}-\frac{2\beta_{1}-2\beta_{2}}{n}$, then
$$\|[b,\ V^{\beta_{2}}(-\Delta+V)^{-\beta_{1}}]f\|_{L^{p_{2}}(\mathbb{R}^{n})}\lesssim\|f\|_{L^{p_{1}}(\mathbb{R}^{n})}.$$
\end{itemize}
\end{theorem}
\begin{proof}
We only prove (i). (ii) can be obtained by duality. Because $\beta_{2}\leq\beta_{1}$, we can decompose the operator $(-\Delta+V)^{-\beta_{1}}V^{\beta_{2}}$ as
$$(-\Delta+V)^{-\beta_{1}}V^{\beta_{2}}=(-\Delta+V)^{\beta_{2}-\beta_{1}}(-\Delta+V)^{-\beta_{2}}V^{\beta_{2}}.$$
Denote by $L^{\beta_{2}-\beta_{1}}$ and $T_{\beta_{2}}$ the operators $(-\Delta+V)^{\beta_{2}-\beta_{1}}$ and $(-\Delta+V)^{-\beta_{2}}V^{\beta_{2}}$, respectively. Then we can get
\begin{eqnarray*}
&&[b,\ (-\Delta+V)^{-\beta_{1}}V^{\beta_{2}}]f(x)\\
&&=[b,\ (-\Delta+V)^{\beta_{2}-\beta_{1}}(-\Delta+V)^{-\beta_{2}}V^{\beta_{2}}]f(x)\\
&&= bL^{\beta_{2}-\beta_{1}}T_{\beta_{2}}f(x)-L^{\beta_{2}-\beta_{1}}T_{\beta_{2}}(b f)(x)\\
&&=bL^{\beta_{2}-\beta_{1}}T_{\beta_{2}}f(x)-L^{\beta_{2}-\beta_{1}}(bT_{\beta_{2}}f(x))+L^{\beta_{2}-\beta_{1}}(bT_{\beta_{2}}f(x))
-L^{\beta_{2}-\beta_{1}}T_{\beta_{2}}(b f)(x)\\
&&=[b,\ L^{\beta_{2}-\beta_{1}}]T_{\beta_{2}}f(x)+L^{\beta_{2}-\beta_{1}}[b, T_{\beta_{2}}]f(x).
\end{eqnarray*}
By (1) of Corollary \ref{coro3.6}, we can get
\begin{eqnarray*}
\Big\|\Big[b,\ (-\Delta+V)^{-\beta_{1}}V^{\beta_{2}}\Big]f\Big\|_{L^{p_{2}}(\mathbb{R}^{n})}
&\lesssim&\Big\|\Big[b,\ L^{\beta_{2}-\beta_{1}}\Big]T_{\beta_{2}}f\Big\|_{L^{p_{2}}(\mathbb{R}^{n})}+\Big\|L^{\beta_{2}-\beta_{1}}\Big[b,\ T_{\beta_{2}}\Big]f\Big\|_{L^{p_{2}}(\mathbb{R}^{n})}\\
&\lesssim&\Big\|T_{\beta_{2}}f\Big\|_{L^{p_{1}}(\mathbb{R}^{n})}+\Big\|\Big[b,\ T_{\beta_{2}}\Big]f\Big\|_{L^{p_{1}}(\mathbb{R}^{n})}\\
&\lesssim&\|f\|_{L^{p_{1}}(\mathbb{R}^{n})}.
\end{eqnarray*}
This completes the proof.
\end{proof}
In the rest of this section, we prove the boundedness of the commutators $[b,V^{\beta_{2}}(-\Delta+V)^{-\beta_{1}}]$ and $[b,(-\Delta+V)^{-\beta_{1}}V^{\beta_{2}}]$ on $L^{p_{2},q,\lambda}_{\alpha,\theta,V}(\mathbb{R}^{n})$, respectively. 
\begin{theorem}\label{th-5.2}
Suppose that $V\in B_{s}$, $s\geq\frac{n}{2}$,
 $\alpha\in(-\infty,0]$, and $\lambda\in(0,n)$. Let $1<q<\infty$, $1<\beta_{2}\leq\beta_{1}<\frac{n}{2}$ and $1<p_{2}<\frac{s}{\beta_{2}}$ with $\frac{1}{p_{1}}-\frac{1}{p_{2}}=\frac{2\beta_{1}-2\beta_{2}}{n}$.
If  $\frac{\lambda}{q}-\frac{1}{p_{1}}+\frac{2\beta_{1}}{n}<\theta<\frac{\lambda}{q}+1-\frac{1}{p_{1}}$, then for $b\in BMO(\mathbb{R}^{n})$,
$$\|[b,V^{\beta_{2}}(-\Delta+V)^{-\beta_{1}}]f\|_{L^{p_{2},q,\lambda}_{\alpha,\theta,V}}\lesssim \|f\|_{L^{p_{1},q,\lambda}_{\alpha,\theta,V}}\|b\|_{BMO}.$$
\end{theorem}
\begin{proof}
For any ball $B(x_{0},r)$, we have
$$f(y)=\sum^{\infty}_{j=-\infty}f(y)\chi_{E_{j}}(y)=\sum^{\infty}_{j=-\infty}f_{j}(y),$$
where $E_{j}=B(x_{0},2^{j}r)\backslash B(x_{0},2^{j-1}r)$. Hence, we have
\begin{eqnarray*}
&&(1+rm_{V}(x_{0}))^{\alpha}r^{-\lambda n}\sum^{0}_{k=-\infty}|E_{k}|^{\theta q}\Big\|\chi_{k}\Big[b,V^{\beta_{2}}(-\Delta+V)^{-\beta_{1}}\Big]f\Big\|^{q}_{L^{p_{2}}(\mathbb{R}^{n})}\\
&&\lesssim (1+rm_{V}(x_{0}))^{\alpha}r^{-\lambda n}\sum^{0}_{k=-\infty}|E_{k}|^{\theta q}\left(\sum^{k-2}_{j=-\infty}\Big\|\chi_{k}\Big[b,V^{\beta_{2}}(-\Delta+V)^{-\beta_{1}}\Big]f_{j}\Big\|_{L^{p_{2}}(\mathbb{R}^{n})}\right)^{q}\\
&&+(1+rm_{V}(x_{0}))^{\alpha}r^{-\lambda n}\sum^{0}_{k=-\infty}|E_{k}|^{\theta q}\left(\sum^{k+1}_{j=k-1}\Big\|\chi_{k}\Big[b,V^{\beta_{2}}(-\Delta+V)^{-\beta_{1}}\Big]f_{j}\Big\|_{L^{p_{2}}(\mathbb{R}^{n})}\right)^{q}\\
&&+(1+rm_{V}(x_{0}))^{\alpha}r^{-\lambda n}\sum^{0}_{k=-\infty}|E_{k}|^{\theta q}\left(\sum^{\infty}_{j=k+2}\Big\|\chi_{k}\Big[b,V^{\beta_{2}}(-\Delta+V)^{-\beta_{1}}\Big]f_{j}\Big\|_{L^{p_{2}}(\mathbb{R}^{n})}\right)^{q}\\
&&=:D_{1}+D_{2}+D_{3}.
\end{eqnarray*}

For $D_{2}$, by (ii) of Theorem \ref{th-bdd-Tbeta}, we have
\begin{eqnarray*}
D_{2}&\lesssim& (1+rm_{V}(x_{0}))^{\alpha}r^{-\lambda n}\sum^{0}_{k=-\infty}|E_{k}|^{\theta q}\left(\sum^{k+1}_{j=k-1}\|f_{j}\|_{L^{p_{1}}(\mathbb{R}^{n})}\right)^{q}\|b\|^{q}_{BMO}\\
&\lesssim& \|f\|^{q}_{L^{p_{1},q,\lambda}_{\alpha,\theta,V}}\|b\|^{q}_{BMO}.
\end{eqnarray*}

For $D_{1}$, by Lemmas \ref{le3.3} \& \ref{le3.2}, we obtain
\begin{eqnarray*}
&&\|\chi_{k}[b,V^{\beta_{2}}(-\Delta+V)^{-\beta_{1}}]f_{j}\|_{L^{p_{2}}(\mathbb{R}^{n})}\\
&&\lesssim \frac{1}{(1+2^{k}rm_{V}(x_{0}))^{N/k_{0}+1}}\frac{1}{(2^{k}r)^{n-2{\beta_{1}}}}
\left(\int_{E_{k}}\Big|\int_{E_{j}}V^{\beta_{2}}(x)(b(x)-b(y))f(y)dy\Big|^{p_{2}}dx\right)^{\frac{1}{p_{2}}}\\
&&\lesssim\frac{1}{(1+2^{k}rm_{V}(x_{0}))^{N/k_{0}+1}}\frac{1}{(2^{k}r)^{n-2{\beta_{1}}}}
\Big[\Big(\int_{E_{k}}V^{\beta_{2}p_{2}}(x)|b(x)-b_{2^{k}r}|^{p_{2}}dx\Big)^{\frac{1}{p_{2}}}\int_{E_{j}}|f(y)|dy\\
&&+\Big(\int_{E_{k}}V^{\beta_{2}p_{2}}(x)dx\Big)^{\frac{1}{p_{2}}}\int_{E_{j}}|b(y)-b_{2^{k}r}||f(y)|dy\Big]\\
&&\lesssim\frac{\|b\|_{BMO}}{(1+2^{k}rm_{V}(x_{0}))^{N/k_{0}+1}}\frac{1}{(2^{k}r)^{n-2{\beta_{1}}}}
\Big[\Big(\int_{E_{k}}V(x)dx\Big)^{\beta_{2}}|E_{k}|^{\frac{1}{p_{2}}-\beta_{2}}\int_{E_{j}}|f(y)|dy\\
&&\quad +\Big(\int_{E_{k}}V(x)dx\Big)^{\beta_{2}}|E_{k}|^{\frac{1}{p_{2}}-\beta_{2}}|E_{j}|^{1-\frac{1}{p_{1}}}(k-j)\|f_{j}\|_{L^{p_{1}}(\mathbb{R}^{n})}\Big]\\
&&\lesssim\frac{\|b\|_{BMO}}{(1+2^{k}rm_{V}(x_{0}))^{N_{1}}}\frac{k-j}{(2^{k}r)^{n-2{\beta_{1}}}}|E_{k}|^{\frac{1}{p_{2}}-\frac{2\beta_{2}}{n}}|E_{j}|^{1-\frac{1}{p_{1}}}
\|f_{j}\|_{L^{p_{1}}(\mathbb{R}^{n})},
\end{eqnarray*}
where $\frac{1}{p_{1}}-\frac{1}{p_{2}}=\frac{2\beta_{1}-2\beta_{2}}{n}$ and $N_{1}<(N/k_{0}+1)-(\log_{2}C_{0}+1)\beta_{2}$. Since $\frac{\lambda}{q}-\frac{1}{p_{1}}+\frac{2\beta_{1}}{n}<\theta<\frac{\lambda}{q}+1-\frac{1}{p_{1}}$, we obtain
\begin{eqnarray*}
D_{1}&\lesssim& \|b\|^{q}_{BMO}(1+rm_{V}(x_{0}))^{\alpha}r^{-\lambda n}\sum^{0}_{k=-\infty}|E_{k}|^{\theta q}\\
&&\times\left(\sum^{k-2}_{j=-\infty}\frac{1}{(1+2^{k}rm_{V}(x_{0}))^{N_{1}}}\frac{k-j}{(2^{k}r)^{n-2{\beta_{1}}}}|E_{k}|^{\frac{1}{p_{2}}-\frac{2\beta_{2}}{n}}|E_{j}|^{1-\frac{1}{p_{1}}}\|f_{j}\|_{L^{p_{1}}(\mathbb{R}^{n})}\right)^{q}\\
&\lesssim& \|b\|^{q}_{BMO}(1+rm_{V}(x_{0}))^{\alpha}r^{-\lambda n}\sum^{0}_{k=-\infty}|E_{k}|^{\theta q}\\
&&\times\left(\sum^{k-2}_{j=-\infty}\frac{(1+2^{j}rm_{V}(x_{0}))^{-\frac{\alpha}{q}}}{(1+2^{k}rm_{V}(x_{0}))^{N_{1}}}\frac{(2^{j}r)^{\frac{\lambda n}{q}}|E_{j}|^{-\theta}}{(2^{k}r)^{n-2\beta_{1}}}\frac{|E_{k}|^{\frac{1}{p_{2}}-\frac{2\beta_{2}}{n}}}{|E_{j}|^{\frac{1}{p_{1}}-1}}(k-j)\right)^{q}
\|f\|^{q}_{L^{p_{1},q,\lambda}_{\alpha,\theta,V}}\\
&\lesssim& \|b\|^{q}_{BMO}\frac{(1+rm_{V}(x_{0}))^{\alpha}}{r^{\lambda n}}\sum^{0}_{k=-\infty}|E_{k}|^{\lambda}\left(\sum^{k-2}_{j=-\infty}(k-j)2^{(j-k)n(\frac{\lambda}{q}-\theta-\frac{1}{p_{1}}+1)}\right)^{q}
\|f\|^{q}_{L^{p_{1},q,\lambda}_{\alpha,\theta,V}}\\
&\lesssim& \|f\|^{q}_{L^{p_{1},q,\lambda}_{\alpha,\theta,V}}\|b\|^{q}_{BMO}.
\end{eqnarray*}

For $D_{3}$, because $\frac{1}{p_{1}}-\frac{1}{p_{2}}=\frac{2\beta_{1}-2\beta_{2}}{n}$ and $N_{1}<(N/k_{0}+1)-(\log_{2}C_{0}+1)\beta_{2}$, we have
\begin{eqnarray*}
&&\|\chi_{k}[b,V^{\beta_{2}}(-\Delta+V)^{-\beta_{1}}]f_{j}\|_{L^{p_{2}}(\mathbb{R}^{n})}\\
&&\lesssim\frac{1}{(1+2^{j}rm_{V}(x_{0}))^{N/k_{0}+1}}\frac{1}{(2^{j}r)^{n-2{\beta_{1}}}}\left(\int_{E_{k}}|\int_{E_{j}}V(x)^{\beta_{2} }(b(x)-b(y))f(y)dy|^{p_{2}}dx\right)^{\frac{1}{p_{2}}}\\
&&\lesssim \frac{j-k}{(1+2^{j}rm_{V}(x_{0}))^{N_{1}}}|E_{j}|^{\frac{2\beta_{1}}{n}-\frac{1}{p_{1}}}|E_{k}|^{\frac{1}{p_{2}}-\frac{2\beta_{2}}{n}}
\|b\|_{BMO}\|f_{j}\|_{L^{p_{1}}(\mathbb{R}^{n})},
\end{eqnarray*}
where we have used the fact that  $|x-y|\sim2^{j}r$ for $x\in E_{k}$, $y\in E_{j}$ and $j\geq k+2$.
Since $\frac{\lambda}{q}-\frac{1}{p_{1}}+\frac{2\beta_{1}}{n}<\theta<\frac{\lambda}{q}+1-\frac{1}{p_{1}}$, we obtain
\begin{eqnarray*}
D_{3}&\lesssim& (1+rm_{V}(x_{0}))^{\alpha}r^{-\lambda n}\sum^{0}_{k=-\infty}|E_{k}|^{\theta q}\\
&&\times\left(\sum^{\infty}_{j=k+2}\frac{j-k}{(1+2^{j}rm_{V}(x_{0}))^{N_{1}}}|E_{j}|^{\frac{2\beta_{1}}{n}-\frac{1}{p_{1}}}|E_{k}|^{\frac{1}{p_{2}}-\frac{2\beta_{2}}{n}}\|b\|_{BMO}\|f_{j}\|_{L^{p_{1}}(\mathbb{R}^{n})}\right)^{q}\\
&\lesssim& \|b\|^{q}_{BMO}(1+rm_{V}(x_{0}))^{\alpha}r^{-\lambda n}\sum^{0}_{k=-\infty}|E_{k}|^{\theta q}\\
&&\times\left(\sum^{\infty}_{j=k+2}\frac{(1+2^{j}rm_{V}(x_{0}))^{-\frac{\alpha}{q}}(2^{j}r)^{\frac{\lambda n}{q}}|E_{j}|^{-\theta}}{(1+2^{j}rm_{V}(x_{0}))^{N_{1}}}\frac{|E_{k}|^{\frac{1}{p_{2}}-\frac{2\beta_{2}}{n}}}{|E_{j}|^{\frac{2\beta_{1}}{n}-\frac{1}{p_{1}}}}(j-k)\right)^{q}
\|f\|^{q}_{L^{p_{1},q,\lambda}_{\alpha,\theta,V}}\\
&\lesssim& \|b\|^{q}_{BMO}(1+rm_{V}(x_{0}))^{\alpha}r^{-\lambda n}\sum^{0}_{k=-\infty}|E_{k}|^{\lambda}\left(\sum^{\infty}_{j=k+2}(j-k)2^{(k-j)n(\theta-\frac{\lambda}{q}+\frac{1}{p_{1}}+\frac{2\beta_{1}}{n})}\right)^{q}
\|f\|^{q}_{L^{p_{1},q,\lambda}_{\alpha,\theta,V}}\\
&\lesssim& \|f\|^{q}_{L^{p_{1},q,\lambda}_{\alpha,\theta,V}}\|b\|^{q}_{BMO}.
\end{eqnarray*}
Let $N$ large enough. Finally, we  get
$$\Big\|\Big[b,\ V^{\beta_{2}}(-\Delta+V)^{-\beta_{1}}\Big]f\Big\|_{L^{p_{2},q,\lambda}_{\alpha,\theta,V}}\lesssim\|f\|_{L^{p_{1},q,\lambda}_{\alpha,\theta,V}}\|b\|_{BMO}.$$
\end{proof}

\begin{theorem}\label{th-5.3}
Suppose that $V\in B_{s}$, $s\geq\frac{n}{2}$ and $b\in BMO(\mathbb{R}^{n})$. Let $\alpha\in (-\infty,0]$, $\lambda\in (0,n)$ and $1<q<\infty.$
If $0<\beta_{2}\leq\beta_{1}<\frac{n}{2}$ , $\frac{s}{s-\beta_{2}}<p_{1}<\frac{n}{2\beta_{1}-2\beta_{2}}$ , $\frac{1}{p_{2}}=\frac{1}{p_{1}}-\frac{2\beta_{1}-2\beta_{2}}{n}$ , $\frac{\lambda}{q}-\frac{1}{p_{2}}<\theta<\frac{\lambda}{q}-\frac{1}{p_{2}}+1-\frac{2\beta_{1}}{n}$, then
$$ \Big\|\Big[b,\ (-\Delta+V)^{-\beta_{1}}V^{\beta_{2}}\Big]f\Big\|_{L^{p_{2},q,\lambda}_{\alpha,\theta,V}}\lesssim \|f\|_{L^{p_{1},q,\lambda}_{\alpha,\theta,V}}\|b\|_{BMO}.$$
\end{theorem}
\begin{proof}
Similarly, we can decompose $f$ based on an arbitrary ball  $B(x_{0},r)$ as follows.
$$f(y)=\sum^{\infty}_{j=-\infty}f(y)\chi_{E_{j}}(y)=\sum^{\infty}_{j=-\infty}f_{j}(y),$$
where $E_{j}=B(x_{0},2^{j}r)\backslash B(x_{0},2^{j-1}r)$. Hence, we have
\begin{eqnarray*}
&&(1+rm_{V}(x_{0}))^{\alpha}r^{-\lambda n}\sum^{0}_{k=-\infty}|E_{k}|^{\theta q}\Big\|\chi_{k}\Big[b,\ (-\Delta+V)^{-\beta_{1}}V^{\beta_{2}}\Big]f\Big\|^{q}_{L^{p_{2}}(\mathbb{R}^{n})}\\
&&\lesssim (1+rm_{V}(x_{0}))^{\alpha}r^{-\lambda n}\sum^{0}_{k=-\infty}|E_{k}|^{\theta q}\left(\sum^{k-2}_{j=-\infty}\Big\|\chi_{k}\Big[b,\ (-\Delta+V)^{-\beta_{1}}V^{\beta_{2}}\Big]f_{j}\Big\|_{L^{p_{2}}(\mathbb{R}^{n})}\right)^{q}\\
&&+(1+rm_{V}(x_{0}))^{\alpha}r^{-\lambda n}\sum^{0}_{k=-\infty}|E_{k}|^{\theta q}\left(\sum^{k+1}_{j=k-1}\Big\|\chi_{k}\Big[b,\ (-\Delta+V)^{-\beta_{1}}V^{\beta_{2}}\Big]f_{j}\Big\|_{L^{p_{2}}(\mathbb{R}^{n})}\right)^{q}\\
&&+(1+rm_{V}(x_{0}))^{\alpha}r^{-\lambda n}\sum^{0}_{k=-\infty}|E_{k}|^{\theta q}\left(\sum^{\infty}_{j=k+2}\Big\|\chi_{k}\Big[b,\ (-\Delta+V)^{-\beta_{1}}V^{\beta_{2}}\Big]f_{j}\Big\|_{L^{p_{2}}(\mathbb{R}^{n})}\right)^{q}\\
&&=F_{1}+F_{2}+F_{3}.
\end{eqnarray*}

 Applying Theorem \ref{th-bdd-Tbeta}, we can get
\begin{eqnarray*}
F_{2}&\lesssim& \frac{(1+rm_{V}(x_{0}))^{\alpha}}{r^{\lambda n}}\sum^{0}_{k=-\infty}|E_{k}|^{\theta q}\left(\sum^{k+1}_{j=k-1}\|f_{j}\|_{L^{p_{1}}(\mathbb{R}^{n})}\right)^{q}\|b\|^{q}_{BMO}\\
&\lesssim& \|f\|^{q}_{L^{p_{1},q,\lambda}_{\alpha,\theta,V}}\|b\|^{q}_{BMO}.
\end{eqnarray*}

For $F_{1}$,  by H\"older's inequality and the fact that $V\in B_{s}$, we apply  Lemmas \ref{le3.2} \& \ref{le3.3}  to deduce that
\begin{eqnarray*}
&&\|\chi_{k}[b,(-\Delta+V)^{-\beta_{1}}V^{\beta_{2}}]f_{j}\|_{L^{p_{2}}(\mathbb{R}^{n})}\\
&&\lesssim\frac{1}{(1+2^{k}rm_{V}(x_{0}))^{N/k_{0}+1}}\frac{1}{(2^{k}r)^{n-2{\beta_{1}}}}\left(\int_{E_{k}}\mid\int_{E_{j}}(b(x)-b(y))V^{\beta_{2}}(y)f(y)dy\mid^{p_{2}}dx\right)^{\frac{1}{p_{2}}}\\
&&\lesssim\frac{1}{(1+2^{k}rm_{V}(x_{0}))^{N/k_{0}+1}}\frac{1}{(2^{k}r)^{n-2{\beta_{1}}}}\Big[\Big(
\int_{E_{k}}|b(x)-b_{2^{k}r}|^{p_{2}}dx\Big)^{\frac{1}{p_{2}}}\int_{E_{j}}|V^{\beta_{2}}(y)f(y)|dy\\
&&+|E_{k}|^{\frac{1}{p_{2}}}\int_{E_{j}}|b(y)-b_{2^{k}r}||V^{\beta_{2}}(y)f(y)|dy\Big]\\
&&\lesssim \frac{(\int_{E_{j}}V(y)dy)^{\beta_{2}}}{(1+2^{k}rm_{V}(x_{0}))^{N/k_{0}+1}}\frac{k-j}{(2^{k}r)^{n-2{\beta_{1}}}}
|E_{k}|^{\frac{1}{p_{2}}}|E_{j}|^{1-\frac{1}{p_{1}}}\|b\|_{BMO}\|f_{j}\|_{L^{p_{1}}(\mathbb{R}^{n})}\\
&&\lesssim \|b\|_{BMO}\frac{k-j}{(1+2^{k}rm_{V}(x_{0}))^{N_{2}}}|E_{k}|^{\frac{1}{p_{2}}+\frac{2\beta_{1}}{n}-1}|E_{j}|^{1-\frac{1}{p_{1}}-\frac{2\beta_{2}}{n}}
\|f_{j}\|_{L^{p_{1}}(\mathbb{R}^{n})},
\end{eqnarray*}
where $\frac{1}{p_{2}}=\frac{1}{p_{1}}-\frac{2\beta_{1}-2\beta_{2}}{n}$ and $N_{2}<(N/k_{0}+1)-(\log_{2}C_{0}+1)\beta_{2}$. Since $\frac{\lambda}{q}-\frac{1}{p_{2}}<\theta<\frac{\lambda}{q}-\frac{1}{p_{2}}+1-\frac{2\beta_{1}}{n}$, we obtain
\begin{eqnarray*}
F_{1}&\lesssim& \|b\|^{q}_{BMO}(1+rm_{V}(x_{0}))^{\alpha}r^{-\lambda n}\sum^{0}_{k=-\infty}|E_{k}|^{\theta q}\\
&&\times\left(\sum^{k-2}_{j=-\infty}\frac{k-j}{(1+2^{k}rm_{V}(x_{0}))^{N_{2}}}|E_{k}|^{\frac{1}{p_{2}}+\frac{2\beta_{1}}{n}-1}|E_{j}|^{1-\frac{1}{p_{1}}-\frac{2\beta_{2}}{n}}\|f_{j}\|_{L^{p_{1}}(\mathbb{R}^{n})}\right)^{q}\\
&\lesssim& \|b\|^{q}_{BMO}(1+rm_{V}(x_{0}))^{\alpha}r^{-\lambda n}\sum^{0}_{k=-\infty}|E_{k}|^{\theta q}\\
&&\left(\sum^{k-2}_{j=-\infty}\frac{(1+2^{j}rm_{V}(x_{0}))^{-\frac{\alpha}{q}}}{(1+2^{k}rm_{V}(x_{0}))^{N_{2}}}
(2^{j}r)^{\frac{\lambda n}{q}}|E_{j}|^{-\theta}\frac{|E_{j}|^{1-\frac{1}{p_{2}}-\frac{2\beta_{1}}{n}}}{|E_{k}|^{1-\frac{1}{p_{2}}-\frac{2\beta_{1}}{n}}}(k-j)
\right)^{q}\|f\|^{q}_{L^{p_{1},q,\lambda}_{\alpha,\theta,V}}\\
&\lesssim& \|b\|^{q}_{BMO}\frac{(1+rm_{V}(x_{0}))^{\alpha}}{r^{\lambda n}}\sum^{0}_{k=-\infty}|E_{k}|^{\lambda}\left(\sum^{k-2}_{j=-\infty}(k-j)2^{(k-j)n(\theta-\frac{\lambda}{q}+\frac{1}{p_{2}}-1+\frac{2\beta_{1}}{n})}\right)^{q}
\|f\|^{q}_{L^{p_{1},q,\lambda}_{\alpha,\theta,V}}\\
&\lesssim& \|f\|^{q}_{L^{p_{1},q,\lambda}_{\alpha,\theta,V}}\|b\|^{q}_{BMO}.
\end{eqnarray*}

For $F_{3}$,  note that when $x\in E_{k}$, $y\in E_{j}$ and $j\geq k+2$, then $|x-y|\sim2^{j}r$. Similar to $F_{1}$, we have
\begin{eqnarray*}
&&\|\chi_{k}[b,(-\Delta+V)^{-\beta_{1}}V^{\beta_{2}}]f_{j}\|_{L^{p_{2}}(\mathbb{R}^{n})}\\
&&\lesssim\frac{1}{(1+2^{j}rm_{V}(x_{0}))^{N/k_{0}+1}}\frac{1}{(2^{j}r)^{n-2{\beta_{1}}}}\left(\int_{E_{k}}\mid\int_{E_{j}}(b(x)-b(y))V(y)^{\beta_{2}}f(y)dy\mid^{p_{2}}dx\right)^{\frac{1}{p_{2}}}\\
&&\lesssim \frac{j-k}{(1+2^{j}rm_{V}(x_{0}))^{N_{2}}}|E_{k}|^{\frac{1}{p_{2}}}|E_{j}|^{-\frac{1}{p_{2}}}\|f_{j}\|_{L^{p_{1}}(\mathbb{R}^{n})}\|b\|_{BMO},
\end{eqnarray*}
where $\frac{1}{p_{2}}=\frac{1}{p_{1}}-\frac{2\beta_{1}-2\beta_{2}}{n}$ and $N_{2}<(N/k_{0}+1)-(\log_{2}C_{0}+1)\beta_{2}$.
Since $\frac{\lambda}{q}-\frac{1}{p_{2}}<\theta<\frac{\lambda}{q}-\frac{1}{p_{2}}+1-\frac{2\beta_{1}}{n}$, we obtain
\begin{eqnarray*}
F_{3}&\lesssim& \|b\|^{q}_{BMO}(1+rm_{V}(x_{0}))^{\alpha}r^{-\lambda n}\sum^{0}_{k=-\infty}|E_{k}|^{\theta q}\\
&&\times\left(\sum^{\infty}_{j=k+2}\frac{(1+2^{j}rm_{V}(x_{0}))^{-\frac{\alpha}{q}}}{(1+2^{j}rm_{V}(x_{0}))^{N_{2}}}(2^{j}r)^{\frac{\lambda n}{q}}|E_{j}|^{-\theta}\frac{|E_{k}|^{\frac{1}{p_{2}}}}{|E_{j}|^{\frac{1}{p_{2}}}}(j-k)\|f_{j}\|_{L^{p_{1}}(\mathbb{R}^{n})}\right)^{q}\\
&\lesssim& \|b\|^{q}_{BMO}(1+rm_{V}(x_{0}))^{\alpha}r^{-\lambda n}\sum^{0}_{k=-\infty}|E_{k}|^{\lambda}\left(\sum^{\infty}_{j=k+2}(j-k)2^{(k-j)n(\theta-\frac{\lambda}{q}+\frac{1}{p_{2}})}\right)^{q}
\|f\|^{q}_{L^{p_{1},q,\lambda}_{\alpha,\theta,V}}\\
&\lesssim& \|f\|^{q}_{L^{p_{1},q,\lambda}_{\alpha,\theta,V}}\|b\|^{q}_{BMO}.
\end{eqnarray*}
Let $N$ large enough. We finally get
$$\|[b,(-\Delta+V)^{-\beta_{1}}V^{\beta_{2}}]f\|_{L^{p_{2},q,\lambda}_{\alpha,\theta,V}}\lesssim\|f\|_{L^{p_{1},q,\lambda}_{\alpha,\theta,V}}\|b\|_{BMO}.$$
\end{proof}

{\bf Competing interests:} The authors declare that they have no competing interests.

{\bf Authors' contributions:}
All authors read and approved the final manuscript.

{\bf Acknowledgements:} Project supported by NSFC No.11171203; New Teacher's Fund for Doctor Stations, Ministry of Education No.20114402120003;
Guangdong Natural Science Foundation S2011040004131; Foundation for
Distinguished Young Talents in Higher Education of Guangdong, China,
LYM11063.

\end{document}